\documentclass[reqno]{amsart}

\usepackage{amsthm,amsmath,amsfonts,amssymb,epsfig,graphicx,latexsym,float,color}
\usepackage{pgfplots}
\usepackage{subfigure}
\usepackage{multirow}
\usepackage{comment}
\usetikzlibrary{plotmarks}
\usetikzlibrary{external}
\newtheorem{definition}{Definition}
\newtheorem{theorem}[definition]{Theorem}
\newtheorem{proposition}[definition]{Proposition}
\newtheorem{corollary}[definition]{Corollary}

\newtheorem{lemma}[definition]{Lemma}

\newcommand{\spec}{\mathrm{spec}}
\usepackage[titlenumbered]{algorithm2e}
\usepackage{paralist}

\usepackage{mathtools}
\usepackage{letltxmacro}
\LetLtxMacro\orgvdots\vdots
\LetLtxMacro\orgddots\ddots

\makeatletter
\DeclareRobustCommand\vdots{%
  \mathpalette\@vdots{}%
}
\newcommand*{\@vdots}[2]{%
  \sbox0{$#1\cdotp\cdotp\cdotp\m@th$}%
  \sbox2{$#1.\m@th$}%
  \vbox{%
    \dimen@=\wd0 %
    \advance\dimen@ -3\ht2 %
    \kern.5\dimen@
    \dimen@=\wd2 %
    \advance\dimen@ -\ht2 %
    \dimen2=\wd0 %
    \advance\dimen2 -\dimen@
    \vbox to \dimen2{%
      \offinterlineskip
      \copy2 \vfill\copy2 \vfill\copy2 %
    }%
  }%
}
\DeclareRobustCommand\ddots{%
  \mathinner{%
    \mathpalette\@ddots{}%
    \mkern\thinmuskip
  }%
}
\newcommand*{\@ddots}[2]{%
  \sbox0{$#1\cdotp\cdotp\cdotp\m@th$}%
  \sbox2{$#1.\m@th$}%
  \vbox{%
    \dimen@=\wd0 %
    \advance\dimen@ -3\ht2 %
    \kern.5\dimen@
    \dimen@=\wd2 %
    \advance\dimen@ -\ht2 %
    \dimen2=\wd0 %
    \advance\dimen2 -\dimen@
    \vbox to \dimen2{%
      \offinterlineskip
      \hbox{$#1\mathpunct{.}\m@th$}%
      \vfill
      \hbox{$#1\mathpunct{\kern\wd2}\mathpunct{.}\m@th$}%
      \vfill
      \hbox{$#1\mathpunct{\kern\wd2}\mathpunct{\kern\wd2}\mathpunct{.}\m@th$}%
    }%
  }%
}
\makeatother

\begin{document}

\title[Energy-Preserving Iteration Schemes for Gauss Collocation Integrators]{Energy-Preserving Iteration Schemes \\ for Gauss Collocation Integrators}
\author[S. Maier]{Stefan Maier$^{1}$}
\author[N. Marheineke]{Nicole Marheineke$^{1}$}
\author[A. Frommer]{Andreas Frommer$^{2}$}
\date{\today\\
$^1$ Universit\"at Trier, FB IV -- Mathematics, Universit\"atsring 15, D-54296 Trier, Germany, \\
$^2$ Bergische Universität Wuppertal, Department of Mathematics, Gaußstraße 20, D-42119 Wuppertal\\
\emph{Acknowledgement:} This work was partially funded by the Deutsche Forschungsgemeinschaft (DFG, German Research Foundation) -- Project-ID 531152215 -- CRC 1701.}
\begin{abstract}
In this work, we develop energy-preserving iterative schemes for the (non-)linear systems arising in the Gauss integration of Poisson systems with quadratic Hamiltonian. Exploiting the relation between Gauss collocation integrators and diagonal Pad\'e approximations, we establish a Krylov-subspace iteration scheme based on a $Q$-Arnoldi process for linear systems that provides energy conservation not only at convergence ---as standard iteration schemes do---, but also at the level of the individual iterates. It is competitive with GMRES in terms of accuracy and cost for a single iteration step and hence offers significant efficiency gains, when it comes to time integration of high-dimensional Poisson systems within given error tolerances. On top of the linear results, we consider non-linear Poisson systems and design non-linear solvers for the implicit midpoint rule (Gauss integrator of second order), using the fact that the associated Pad\'e approximation is a Cayley transformation.  For the non-linear systems arising at each time step, we propose fixed-point and Newton-type iteration schemes that inherit the convergence order with comparable cost from their classical versions, but have energy-preserving iterates.
\end{abstract}
\maketitle

\noindent
{\sc Keywords.} Poisson system; Krylov-subspace iteration scheme; Pad\'e approximation; Cayley transformation; Newton-like method; Fixed-point iteration \\
{\sc AMS-Classification.} 37Mxx; 47J25; 65Lxx; 65Pxx  
\section{Introduction}
This article deals with efficient energy-preserving solvers for the (non-)linear systems arising in implicit numerical integration schemes for Poisson systems of the form  
\begin{align}
    \label{eq:energypresSys}
    \dot{y} = J(y)\nabla \mathcal{H}(y), \quad y(0) = y_0
\end{align} 
with state $y\in \mathcal{C}^1([0, T], \mathbb{R}^n)$, skew-symmetric structure matrix  $J(y)^T = -J(y)\in \mathbb{R}^{n\times n}$ for all $y \in \mathbb{R}^{n}$ and quadratic Hamiltonian $\mathcal{H}(y) = \tfrac{1}{2}y^TQy$ with $Q^T = Q \in \mathbb{R}^{n \times n}$, $Q > 0$. The Hamiltonian $\mathcal{H}$ is an invariant of the system, i.e.\ every solution $y(t)$ of \eqref{eq:energypresSys} fulfills $\mathcal{H}(y(t)) = 
\mathcal{H}(y_0)$ for all $t \in [0,T]$. In physical models, $\mathcal{H}$ is usually an energy functional, such that the invariance can be interpreted as conservation of energy. The Poisson system \eqref{eq:energypresSys} is an autonomous non-linear ordinary differential equation (ODE) which becomes linear if $J \in \mathbb{R}^{n \times n}$ is constant, i.e.
\begin{align}
\label{eq:energypresSysLinQ}
\dot{y} = JQy, \quad y(0) = y_0.
\end{align}

Poisson systems arise in the modeling of energy exchange, \cite{Hairer2004GeometricNI,book_poisson_structures}. Special types are the Lie-Poisson systems where the structure matrix depends linearly on the state, and the Hamiltonian systems where the structure matrix is given by
$J = \left(\begin{smallmatrix}
    0 & I\\
    -I & 0 
    \end{smallmatrix}\right)$ with identity matrix $I$, see, e.g., \cite{arnol2013mathematical, McLachlan_1998}. Poisson systems also appear in the port-Hamiltonian modeling framework \cite{PortHamIntrodcOverview} when splitting schemes with energy-based decomposition are applied, \cite{phDAE,frommer2023operatorSplitting,monch2024commutatorbasedoperatorsplittinglinear}. Various numerical methods have been developed for time integration that preserve the invariants which characterize properties such as symplectic structures and energy-preservation. An important class of numerical integrators that preserve quadratic invariants as given in \eqref{eq:energypresSys} are the Gauss methods. Gauss integrators are collocation schemes belonging to the group of implicit A-stable Runge-Kutta methods, \cite{Hairer2004GeometricNI}. Discrete gradient methods are particularly suitable for more general (non-quadratic) Hamiltonians \cite{mclachlan1999geometric, norton2013discretegradientmethodspreserving}. Projection-based approaches use explicit integration schemes, such as explicit Runge-Kutta integrators, and then project the numerical solution onto the manifold characterized by the invariants by solving a minimization problem, \cite{A_Norton_2015}. A way to look at invariants is to consider a manifold based on a Lie group action; for a deeper insight into the theory we refer to \cite{olver1986introduction, varadarajan2013lie}. This leads to the class of Lie group integrators, such as Crouch-Grossmann methods \cite{jackiewiczCrouchGrossmanMethods} and Munthe-Kaas Runge-Kutta methods \cite{crouch1993numerical,MUNTHEKAAS1999115}, where local parameterizations are used. In this context we highlight the Cayley transformation as a local parameterization for quadratic Lie groups, see, e.g., \cite{lewis1994conserving,LopezApplCayleyapproach,wandeltGeoIntLieGrCayleyTrans}. 
    
In this paper, we consider the well-established Gauss collocation integrators for the Poisson system \eqref{eq:energypresSys} which preserve its structure in the sense of energy conservation. Since the integrators require to solve a (non-)linear system of equations at each time step, the computation can become very expensive ---especially in case of high-dimensional problems or high-order approximations. For this reason, we aim to develop iterative schemes that handle the resulting systems of equations particularly efficiently. We investigate the relation between Gauss methods and Pad\'e approximations
and derive a Krylov-subspace based iterative solver for linear systems, our so-called \emph{$Q$-Arnoldi approximation}, which preserves the energy at each iterative step. Hence, we can terminate the iteration as soon as we have reached a desired accuracy, which should usually be consistent with resolution and order of the used Gauss integrator. In this way, we achieve significant efficiency gains over standard iteration schemes, such as GMRES \cite{KrylovPrinciplesAnalysis,SogabeKrylovSubspaceMethods}, that cause an accumulative deviation from energy conservation if not run until full convergence at each time step. On top of the linear results, we design non-linear solvers for the implicit midpoint rule (Gauss integrator of order $p=2$). We particularly make use of the fact that the associated Pad\'e approximation is a Cayley transformation. The resulting schemes of fixed-point type as well as of Newton type are competitive with the classical fixed-point iteration and Newton-like methods in terms of convergence behavior and cost for a single iteration step, but superior in terms of energy preservation: energy is not only conserved upon convergence, but also at the level of each individual iterate. 

This paper is structured as follows: We introduce Pad\'e approximations and discuss their invariance properties in a Lie group / Lie algebra setting in Section~\ref{sec:pade}. In Section~\ref{krylovsubSec} we exploit the relation between Gauss collocation integrators and diagonal Pad\'e approximations of the exponential function and establish a Krylov-subspace iteration scheme based on a $Q$-Arnoldi process that provides energy-preserving iterates when solving linear Poisson systems. A performance study in comparison to GMRES is provided in Section~\ref{sec:num_lin}. Section~\ref{sec:nonlinear} presents the extension to non-linear Poisson systems. We set up energy-preserving non-linear solvers of linear and superlinear convergence order for the implicit midpoint rule. Their performance is studied in Section~\ref{sec:num_nl}.

\section{Padé approximations} \label{sec:pade}

Padé approximations are rational approximations to functions obtained by matching derivatives at the origin, cf.\ \cite{Baker_Graves-Morris_1996,numericalAnaIntro}. The diagonal degree $s$ Padé approximation $\mathcal{R}_s(z)$ to the exponential function is given by 
\begin{align}
\label{eq:diagonalePadé}
    \mathcal{R}_s(z) = \frac{\mathcal{D}_{s}(z)}{\mathcal{D}_{s}(-z)}, \quad \mathcal{D}_{s}(z) = \sum_{j = 0}^s \frac{s! (2s-j)!}{(2s)!j!(s-j)!}z^j, \enspace z \in \mathbb{C}, \enspace \mathcal{D}_s(-z) \neq 0.
\end{align}
The following is known about $\mathcal{R}_s(z)$; see \cite{numericalAnaIntro}.
\begin{proposition} \label{prop:poles}
All poles $\tau_j$ of $\mathcal{R}_s(z)$, i.e.\ the zeros of $\mathcal{D}_s(-z)$, are simple and located in the open right half plane, i.e.\ $\Re(\tau_j) > 0$. Moreover, in the partial fractional decomposition of $\mathcal{R}_s(z)$,
\begin{align} \label{eq:pfd}
   \mathcal{R}_s(z) = \frac{\mathcal{D}_{s}(z)}{\mathcal{D}_{s}(-z)} = (-1)^s + \sum_{j = 1}^{s} \frac{\omega_j}{z - \tau_j},
\end{align}
the poles $\tau_j$ and coefficients $\omega_j$ appear in complex conjugate pairs or are real-valued. 
\end{proposition}

If $A \in \mathbb{R}^{n \times n}$ is a matrix whose spectrum $\spec(A)$ has empty intersection with the poles $\tau_j$ of $\mathcal{R}_s(-z)$ ---which is equivalent to $\mathcal{D}_s(-A)$ being non-singular---,  we can extend the diagonal Padé approximation to a (matrix) function of $A$ in the usual sense, which yields
\begin{align}
\label{eq:extDiagonalPadé}
    \mathcal{R}_s(A) = \mathcal{D}_s(-A)^{-1}\mathcal{D}_s(A) = \mathcal{D}_s(A)\mathcal{D}_s(-A)^{-1} = (-I)^s + \sum_{j = 1}^{s} \omega_j (A - \tau_jI)^{-1}.
\end{align}
Note that these equalities hold because identities between scalar functions on $\mathbb{C}$ generate similar identities for the corresponding matrix functions; see \cite{matrixFunctBook,lancaster1985theory} for details. The matrix extension of the diagonal Padé approximation \eqref{eq:extDiagonalPadé} has important invariance properties. To discuss these, we introduce the quadratic Lie group $\mathcal{G}_Q$ and the associated Lie algebra $\mathfrak{g}_Q$, which for fixed $Q \in \mathbb{R}^{n \times n}$ are given by 
\[
\mathcal{G}_Q = \{ B \in \mathbb{R}^{n \times n}  :  B^TQB = Q \}, \quad \mathfrak{g}_Q = \mathfrak{T}_I\mathcal{G}_Q = \{ A \in \mathbb{R}^{n \times n}  :  Q A + A^TQ = 0 \},
\]
where $\mathfrak{T}_I$ is the tangent space at the identity $I$.
The case of $Q$ being symmetric positive definite, $Q > 0$, is of particular interest. Then $Q$ induces the inner product $\langle x,y\rangle_Q:= \langle Qx,y\rangle$ on $\mathbb{R}^n$, and $\mathfrak{g}_Q$ consists of exactly those matrices $A$ which are anti-selfadjoint with respect to this inner product, i.e.\
\[
\langle Ax,y\rangle_Q = - \langle x,Ay \rangle_Q \enspace \text{for all } x,y \in \mathbb{R}^{n}.
\]  
For future use, we observe the following properties of such $\mathfrak{g}_Q$ and $\mathcal{G}_Q$.

\begin{lemma} \label{lem:lie_algebra_spec} Let $A\in \mathfrak{g}_Q$, $Q >0$. Then all eigenvalues of $A$ are purely imaginary, i.e.\ $\spec(A) \subset i\mathbb{R}$.
\end{lemma}
\begin{proof} For $\lambda \in \spec(A)$ with eigenvector $x$ we have 
\[
\lambda \langle Qx,x \rangle = \langle QAx,x \rangle = \langle x,A^TQx \rangle = \langle x, -QAx \rangle = -\overline{\lambda} \langle x, Qx \rangle  = -\overline{\lambda} \langle Qx, x \rangle.
\]
Since $\langle Qx,x\rangle \neq 0$, this implies $\lambda = - \overline{\lambda}$, which means $\lambda \in i\mathbb{R}$.
\end{proof}

We now turn to the announced invariance properties.

\begin{theorem}
\label{th:LieGroupAlgebraPade}
For the diagonal Pad\'e approximations $\mathcal{R}_s$ we have
\begin{itemize}
\item[(i)] If $A \in \mathfrak{g}_Q$ for some $Q$ and $\mathcal{R}_s(A)$ is defined, then $\mathcal{R}_s(A) \in \mathcal{G}_Q$.
\item[(ii)] If $A \in \mathfrak{g}_Q$ and $Q>0$, then $\mathcal{R}_s(A)$ is always defined, i.e.\ $\mathcal{R}_s$ maps all of $\mathfrak{g}_Q$ on $\mathcal{G}_Q$.
\end{itemize}
\end{theorem}
\begin{proof}
Let $A \in \mathfrak{g}_Q$. To prove part (i) we observe that we have $QA^n = ((-A)^n)^TQ$ for all $n \in \mathbb{N}$, and thus 
    \begin{align*}
       Q\mathcal{D}_s(A) = Q \sum_{i= 0}^s a_i A^i = \sum_{i= 0}^s a_i QA^i = \sum_{i= 0}^s a_i ((- A)^i)^TQ = \mathcal{D}_s(-A)^TQ.
    \end{align*}
    Similarly, we also have $Q\mathcal{D}_s(-A) = \mathcal{D}_s(A)^TQ$. Thus, if $\mathcal{D}_s(A)$ is non-singular, i.e.\ $\mathcal{R}_s(A)$ is defined, noting that $\mathcal{D}_s(-A)^{-1}$ and $\mathcal{D}_s(A)$ commute, these two equalities imply $B^TQB = Q$ for $B = \mathcal{R}_s(A) = \mathcal{D}_s(-A)^{-1}\mathcal{D}_s(A)$.

To prove part (ii) note that by Lemma~\ref{lem:lie_algebra_spec} the spectrum of $A$ is purely imaginary, whereas by Proposition~\ref{prop:poles} the poles of $\mathcal{R}_s$ lie in the open right half plane. So the two sets have empty intersection.
\end{proof}

Let us remark that the proof shows that part (i) of Theorem~\ref{th:LieGroupAlgebraPade} holds for general rational functions of the form $\mathcal{P}(z)/\mathcal{P}(-z)$ with $\mathcal{P}$ being a polynomial with real coefficients, and part (ii) then holds if the zeros of $\mathcal{P}(-z)$ lie off the imaginary axis.  

\begin{theorem}
\label{th:PadeLip}
    Let $Q > 0$. The diagonal Pad\'e approximations $\mathcal{R}_s(A)$ are (globally) Lipschitz continuous on $\mathfrak{g}_Q$. With respect to the matrix norm $\| \cdot \|_Q$  associated with the $Q$-norm $\langle x,x\rangle_Q^{1/2}$ on $\mathbb{R}^{n}$, 
    a Lipschitz constant for each $s$ is 
    \begin{equation} \label{eq:lip_constant}
    \kappa_s = \sum_{j=1}^s \frac{|\omega_j|}{\Re(\tau_j)^2}.
    \end{equation}
\end{theorem}
\begin{proof} Using the partial fractional decomposition \eqref{eq:pfd}, we have for $A,B \in \mathfrak{g}_Q$:
\begin{align*}
\mathcal{R}_s(A)-\mathcal{R}_s(B) 
&= \sum_{j=1}^s \omega_j \left( (A-\tau_j I)^{-1} -(B-\tau_j I)^{-1} \right) \\ 
&=  \sum_{j=1}^s \omega_j (B-\tau_j I)^{-1} \left( (B-\tau_j I) - (A-\tau_j I) \right) (A-\tau_j I)^{-1}  \\
&= \sum_{j=1}^s \omega_j (B-\tau_j I)^{-1} \left( B- A \right) (A-\tau_j I)^{-1},
\end{align*}
which implies
\[
\|\mathcal{R}_s(A)-\mathcal{R}_s(B)\|_Q \leq 
\sum_{j=1}^s |\omega_j| \, \| (A-\tau_j I)^{-1} \|_Q \,\|(B-\tau_j I)^{-1}\|_Q \cdot \|A- B\|_Q. 
\]
The proof will be finished once we have shown $\|(A-\tau_j I)^{-1}\|_Q, \|(B-\tau_j I)^{-1}\|_Q \leq \frac{1}{\Re(\tau_j)}$. 
This can be seen as follows (we just consider $A$): Since $A$ is skew-selfadjoint w.r.t.\ the $Q$-inner product, there is a basis of $\mathbb{C}^{n}$ of $Q$-orthogonal eigenvectors of $A$, and the same holds for 
$A-\tau_j I$ and its inverse. This implies that the matrix norm is given by the maximum modulus of the eigenvalues. Since $\spec(A) \subset i\mathbb{R}$, the smallest modulus of an eigenvalue of $A-\tau_j I$ is bounded from below by $\Re(\tau_j)$, and since the eigenvalues of  $(A-\tau_j I)^{-1}$ are the inverses of those of  $A-\tau_j I$, the largest modulus of an eigenvalue of $(A-\tau_j I)^{-1}$ is bounded from above by $1/\Re(\tau_j)$. This  finishes the proof.    
\end{proof}

\begin{corollary}  \label{cor:Lipschitz} For $Q >0$, the first diagonal Pad\'e approximation $\mathcal{R}_1$ is Lipschitz-continuous on $\mathfrak{g}_Q$ with Lipschitz constant $\kappa_1 = 1$.
\end{corollary}
\begin{proof} We have $\mathcal{R}_1(z) = \frac{1+0.5z}{1-0.5z} = -1 - \frac{4}{z-2}$ and thus $\omega_1=-4$ and $\Re(\tau_1) = 2$, which gives $\kappa_1 = 1$ in \eqref{eq:lip_constant}.
\end{proof}

Let us mention at this point that $\mathcal{R}_1$ is related to the Cayley transform
\begin{align} \label{def:cayleyTrans}
\mathcal{C}(A) = (I - A)^{-1}(I + A)
\end{align}
which is defined if $1\not \in \spec(A)$, since we have
 $\mathcal{R}_1(A) = \mathcal{C}(\tfrac{1}{2}A)$. Corollary~\ref{cor:Lipschitz} shows that for $Q >0$, the Cayley transform is Lipschitz continuous on $\mathfrak{g}_Q$ with Lipschitz constant 2.

\section{Krylov-subspace iteration schemes} \label{sec:krylov}
\label{krylovsubSec}

The stability function of the Gauss collocation integrator of order $2s$ is the diagonal Padé approximation of the exponential function of degree $s$; see \cite{hairer_siffODEquation}. This means that for the linear ODE \eqref{eq:energypresSysLinQ}, the Gauss collocation approximation $y_1$ to $y(h)$, the solution at time $h$, is defined as 
\begin{align} \label{eq:refGauss}
y_1 = \mathcal{R}_s(hJQ)y_0 = \mathcal{D}_s(-hJQ)^{-1}\mathcal{D}_s(hJQ)y_0,
\end{align}
which, using the partial fractional decomposition \eqref{eq:pfd}, is equivalent to computing 
\begin{align}
    \label{eq:refGauss_pdf}
    y_1 = 
    (-1)^sy_0 + \sum_{j = 1}^{s} \omega_j (hJQ - \tau_jI)^{-1}y_0.
\end{align}
Gauss collocation integrators are particularly appropriate in our context, since they conserve the $Q$-norm (the ``energy''), just as the exact solution of the Poisson system \eqref{eq:energypresSysLinQ} does. Indeed, since $hJQ \in \mathfrak{g}_Q$, Theorem~\ref{th:LieGroupAlgebraPade} gives $\mathcal{R}_s(hJQ) \in \mathcal{G}_Q$ and thus the Hamiltonian fulfills
\[
\mathcal{H}(y_1) = \tfrac{1}{2}y_1^TQy_1 = \tfrac{1}{2}y_0^T\mathcal{R}_s(hJQ)^TQ\mathcal{R}_s(hJQ)y_0 = \tfrac{1}{2}y_0^TQy_0 = \mathcal{H}(y_0).
\]

If the system dimension is large, computing $y_1$ exactly via \eqref{eq:refGauss} or \eqref{eq:refGauss_pdf} is computationally costly because it involves solving large linear systems. Our goal is now to devise efficient schemes which approximate the solution $y_1$ of the linear systems iteratively. If these iterations can be stopped relatively early, e.g., when an accuracy compatible with the order of the integrator is achieved, this will result in relatively low cost. It is important then that the $Q$-norm is also conserved on the level of the individual iterates of these linear solvers, thus preserving the prominent invariant of the underlying ODE. In the following we derive a solver fulfilling all these requirements within the Krylov subspace framework. For a good overview of Krylov subspace theory, we refer to \cite{KrylovPrinciplesAnalysis,SogabeKrylovSubspaceMethods}.      

\subsection{Krylov subspaces and the Arnoldi approximation}
 \label{subsec:KrylovArnoldi}
 
For given $A \in \mathbb{R}^{n \times  n}$ and $v \in \mathbb{R}^n$, $v \neq 0$, the $k$-th Krylov subspace $\mathcal{K}_k(A, v)$ is given by 
\[
\mathcal{K}_k(A, v) = \text{span} \{ v, Av, \dots, A^{k-1}v\}.
\]
For $Q \in \mathbb{R}^{n \times n}$, $Q>0$,
we iteratively can construct a $Q$-orthonormal basis of the Krylov subspace via the $Q$-Arnoldi process as given in Algorithm \ref{alg:arnoldi_P}. Putting $V_k = [v_1| \cdots | v_k] \in \mathbb{R}^{n\times k}$ and $H_k\in \mathbb{R}^{k \times k}$ the upper Hessenberg matrix with the orthogonalization coefficients $h_{ij}$, 
the $Q$-Arnoldi process can be summarized as yielding the Arnoldi relation
\begin{equation} \label{eq:Arnoldi_relation}
AV_k = V_kH_k + h_{k+1,k} v_{k+1} e_k^T, \qquad V_k^TQV_k = I,
\end{equation}
where $e_k$ denotes the $k$-th canonical unit vector in $\mathbb{R}^k$.
An immediate consequence of the $Q$-orthonormality of $V_k$ is that for any $\tau \in \mathbb{C}$ we have
\begin{equation} \label{eq:projected_Arnoldi_relation}
V_k^TQ(A-\tau I)V_k = H_k-\tau I. 
\end{equation}

\begin{algorithm}[!t]
    \label{alg:arnoldi_P}
    \TitleOfAlgo{$Q$-Arnoldi process} 
    \DontPrintSemicolon
    $v_1 \gets \frac{v}{\lVert v \rVert_Q}$\;
    \For{$k = 1,2, \dots$}
    {
        $w_k \gets Av_k$\;
        \For{$i = 1, \dots, k$}
        {
            $h_{i,k} \gets \langle v_i, w_k \rangle_Q$\;
            $w_k \gets w_k - h_{i,k}v_i$\;
        }
        $h_{k+1,k} \gets \lVert w_k \rVert_Q$. If $h_{k+1,k} = 0$, Stop\;
        $v_{k+1} = \frac{w_k}{h_{k+1,k}}$\;
    }
\DontPrintSemicolon
\end{algorithm}

A common approach to approximate the solution of a linear system $(A-\tau I)x = v$ with iterates $x_k(\tau) = V_k\xi_k(\tau)$ taken from $\mathcal{K}_k(A, v)$  is to use a Galerkin condition which requires that the residual $v-(A-\tau I)x_k(\tau)$ is $Q$-orthogonal to $\mathcal{K}_k(A, v)$.  
Using the Arnoldi relation \eqref{eq:projected_Arnoldi_relation}, the Galerkin condition results in 
\begin{eqnarray*}
v-(A-\tau I)V_k\xi_k(\tau) \perp_Q \mathcal{K}_k(A,v) 
& \Longleftrightarrow & 
V_k^TQ(v-(A-\tau I)V_k\xi_k(\tau)) = 0 \\
& \Longleftrightarrow & V_k^TQv - (H_k -\tau I) \xi_k(\tau) = 0.
\end{eqnarray*}
Since $V_k^TQv = e_1\|v\|_Q$ with $e_1$ the first canonical unit vector in $\mathbb{R}^k$, this gives 
\[
x_k(\tau) =  V_k(H_k-\tau I)^{-1}  e_1 \lVert v \rVert_Q.
\]

Using this approach for each term in the action of the partial fractional decomposition \eqref{eq:extDiagonalPadé} (Pad\'e approximation) on a vector $v$, we obtain what we might call the {\em $Q$-Arnoldi approximation} of $\mathcal{R}_s(A)v$ given as
\begin{align}
x_k 
    &= \underbrace{(-1)^sv}_{= (-1)^s V_k e_1 \lVert v \rVert_Q} + \sum_{j = 1}^s  V_k\omega_j (H_k - \tau_jI)^{-1}e_1 \lVert v \rVert_Q 
    \label{eq:Arnoldi_approx_pfd} \\
    &= V_k \mathcal{R}_s(H_k)e_1 \lVert v \rVert_Q \nonumber \\
    & =   V_k \mathcal{D}_{s}(-H_k)^{-1}\mathcal{D}_{s}(H_k)e_1 \lVert v \rVert_Q .
    \label{eq:Arnoldi_approx_rational}
\end{align}
For the standard inner product ($Q = I$), this is a special case of the Arnoldi matrix function approximation for $f(A)v$, defined as $V_kf(H_k) e_1 \|v\|$; see, e.g., \cite{FrSi06,matrixFunctBook,analysisProjMethod}.

For $A \in \mathfrak{g}_Q$, we observe the following two crucial properties of the $Q$-Arnoldi process and the $Q$-Arnoldi approximation in \eqref{eq:Arnoldi_approx_pfd}, \eqref{eq:Arnoldi_approx_rational}; see also \cite{diab2022flexibleshortrecurrencekrylov,GLMS2022}.

\begin{theorem} \label{prop:Arnoldi_iterates_invariance} Let $A \in \mathfrak{g}_Q$, $Q>0$. Then
\begin{itemize}
\item[(i)] The upper Hessenberg matrix $H_k$ from the $Q$-Arnoldi process (Algorithm~\ref{alg:arnoldi_P}) is tridiagonal and skew-symmetric, $H_k^T = - H_k$.
\item[(ii)] The $Q$-Arnoldi approximation $x_k$  from \eqref{eq:Arnoldi_approx_pfd}, or, equivalently, \eqref{eq:Arnoldi_approx_rational} is defined, i.e.\ all matrices $H_k-\tau_j I$ are non-singular, and it satisfies 
\[
\|x_k\|_Q = \|v\|_Q \text{ for all $k$.}
\]
\end{itemize}
\end{theorem}
\begin{proof} We have $H_k = V_k^TQAV_k$
by \eqref{eq:Arnoldi_relation}, and $QA = -A^TQ$, since $A \in \mathfrak{g}_Q$. Thus, $H_k = V_k^TQAV_k = -V_k^TA^TQV_k = -H_k^T$, which proves that $H_k$ is not only upper Hessenberg, but even tridiagonal, and that it is, in addition, skew-symmetric. So all eigenvalues of $H_k$ are purely imaginary, and those of $H_k - \tau_j I$ all have real part $-\tau_j <0$. Therefore, $H_k - \tau_j I$ is non-singular. And since $H_k \in \mathfrak{g}_I$, Theorem~\ref{th:LieGroupAlgebraPade} (i) shows that $\mathcal{R}_s(H_k) \in \mathcal{G}_I$, i.e.\ $ \mathcal{R}_s(H_k)^T \mathcal{R}_s(H_k) =I$. Applied to the Arnoldi approximation, this gives
\[
\|x_k\|_Q^2 \, = \, \left( V_k \mathcal{R}_{s}(H_k) e_1 \lVert v \rVert_Q\right)^T Q  V_k \mathcal{R}_{s}(H_k)e_1 \lVert v \rVert_Q \, =  \, 
e_1^T\mathcal{R}_{s}(H_k)^TV_k^TQV_k\mathcal{R}_{s}(H_k)e_1 \lVert v\rVert_Q^2 
\, = \, 
\lVert v \rVert_Q^2.
\]
\end{proof}
Now, consider that we approximate the result $y_1$ of the degree $2s$ Gauss collocation integrator for the Poisson system given in \eqref{eq:refGauss},
\[
y_1 = \mathcal{R}_s(hJQ)y_0,
\]
by the $k$-th $Q$-Arnoldi approximation
\[
x_k = V_k \mathcal{R}_{s}(H_k) e_1 \lVert y_0 \rVert_Q \in \mathcal{K}_k(hJQ,y_0), \qquad x_k \rightarrow y_1 \text{ as } k\rightarrow n
\]
Then part (ii) of Theorem~\ref{prop:Arnoldi_iterates_invariance} yields that these $Q$-Arnoldi iterates satisfy the sought-after invariance
\[
\mathcal{H}(x_k) = \frac{1}{2}x_k^TQx_k = \frac{1}{2}y_0^TQy_0 = \mathcal{H}(y_0)
\]
for any iteration $k$.  Part (i) shows that we can compute the $Q$-Arnoldi iterates efficiently, because the $Q$-Arnoldi process,  Algorithm~\ref{alg:arnoldi_P}, actually reduces to its Lanczos variant, where in the inner loop (over $i$) we only have to compute $h_{k+1,k}$, since $h_{k-1,k} = -h_{k,k-1}$ and $h_{i,k} = 0$ for all other $i$.

The nominator-denominator based representation \eqref{eq:Arnoldi_approx_rational} for $x_k$ and the mathematically equivalent partial fractional decomposition based representation  \eqref{eq:Arnoldi_approx_pfd} yield two algorithmically different approaches to compute $x_k$. We now present and discuss them as variants V.1 and V.2. The common part for both variants is to run the $Q$-Arnoldi process, thus obtaining $H_k$ and $V_k$. 
\begin{enumerate}[\bfseries{V}.1]
    \item \label{enum:a} We evaluate the action of the nominator and of the denominator polynomial of $\mathcal{R}_s$ on a vector as such, i.e.\ 
    \begin{eqnarray*} 
    & &\text{compute } \eta_k = \mathcal{D}_s(H_k)e_1 \text{ (we may use Horner's scheme)}, \\
    & &\text{compute the matrix } D_k = \mathcal{D}_s(-H_k) \text{ (we may use Horner's scheme)}, \\
    & &\text{solve } D_k\xi_k = \eta_k \text{ and put } x_k = V_k \xi_k \|y_0\|_Q.
    \end{eqnarray*}
    This approach requires only one linear solve of a system of size $k$ with matrix $D_k$. The matrix $D_k$ must be computed explicitly, and it has semi-bandwidth $s$, a fact that we can exploit when computing its (pivoted) LU-factorization. 
    \item \label{enum:b} We solve the $s$ linear systems
    \[
    (H_k-\tau_j I)\zeta_j = e_1, \enspace j=1,\ldots s,
    \]
    and then put $x_k = (-1)^s y_0 + V_k \big( \sum_{j=1}^s \omega_j \zeta_j\big) \|y_0\|_Q$. This requires $s$ tridiagonal linear systems of size $k$ to be solved, and, as opposed to V.\ref{enum:a}, it may require complex arithmetic. Savings are possible by using the fact that poles $\tau_j$ and weights $\omega_j$ come in complex conjugate pairs. This reduces the computation to solving only $\lceil \tfrac{s}{2} \rceil$ systems, as we will discuss further below. 
\end{enumerate}

\subsection{Gauss integrators up to order six}
\label{subsec:gausstosix}

We end this section giving some more details for the Gauss integrators of order two (the midpoint rule), four  and six. 

The midpoint rule ($s=1$) prescribes the approximation of the linear ODE~\eqref{eq:energypresSysLinQ} via
\[
\mathcal{R}_1(z) = \frac{1+\tfrac{1}{2}z}{1-\tfrac{1}{2}z} = -1 -\frac{4}{z-2},
\]
cf.\ \eqref{eq:refGauss} and \eqref{eq:refGauss_pdf}. This means that in the $Q$-Arnoldi process, variant V.\ref{enum:a} we compute
\[
x_k = V_k\xi_k \lVert y_0 \rVert_Q, \enspace \text{where } (I - \tfrac{1}{2}H_k)\xi_k = (I + \tfrac{1}{2}H_k)e_1=\eta_1,
\]
whereas V.\ref{enum:b} obtains $x_k$ as
\[
x_k = -y_0 -4V_k\zeta_1 \lVert y_0 \rVert_Q, \enspace \text{where } (H_k - 2I)\zeta_1 = e_1 .
\]
With respect to computational cost, the difference in both versions is marginal: none needs complex arithmetic, both need to solve one tridiagonal system. 

For the fourth order Gauss integrator ($s=2$), we have
\[
\mathcal{R}_2(z) = \frac{1+\tfrac{1}{2}+\tfrac{1}{12}z^2}{1-\tfrac{1}{2}z+\tfrac{1}{12}z^2} = 1 + \frac{6+6\sqrt{3}i}{z-(3-\sqrt{3}i)} + \frac{6-6\sqrt{3}i}{z-(3+\sqrt{3}i)}.
\]
In variant V.1 we compute the vector $\eta_k = (I+\tfrac{1}{2}H_k+ \tfrac{1}{12}H_k^2)e_1$ and the matrix $D_k = I- \tfrac{1}{2}H_k+ \tfrac{1}{12}H_k^2$,  which is pentadiagonal, and we then solve a linear system 
\[
D_k \xi_k =  \eta_k.
\] 
In variant V.\ref{enum:b}, we, in principle, have to deal with two tridiagonal linear systems
\[
\left(H_k - (3-\sqrt{3}i)I\right)\zeta_1 = e_1, \qquad \left(H_k - (3+\sqrt{3}i)I\right)\zeta_2 = e_1.
\]
But since the two summands in $V_k \big( (6+6\sqrt{3}i)\zeta_1 + (6-6\sqrt{3}i)\zeta_2\big)\|y_0\|_Q$ are complex conjugates of each other, we get
\[
x_k = y_0 + 2V_k \Re[(6+6\sqrt{3}i)\zeta_1] \|y_0\|_Q.
\]
Hence, we actually must solve only one (complex) tridiagonal linear system.

The Pad\'e approximation corresponding to the sixth order Gauss integrator ($s=3$) is given by
\[
\mathcal{R}_3(z) = \frac{1+\tfrac{1}{2}z+\tfrac{1}{10}z^2+\tfrac{1}{120}z^3}{1-\tfrac{1}{2}z+\tfrac{1}{10}z^2-\tfrac{1}{120}z^3} = -1 + \frac{\omega_1}{z-\tau_1} + 
 \frac{\overline{\omega}_1}{z-\overline{\tau}_1} + \frac{\omega_2}{z-\tau_2}, 
\]
where the poles and coefficients are numerically determined as
\begin{align*}
\tau_1 &\approx 3.67781464537391 - 3.50876191956744i, 
\qquad &&\tau_2 \approx 4.64437070925217, \\
\omega_1 &\approx 16.6012701235744 - 20.5831842793869i, 
\qquad &&\omega_2 \approx -57.2025402471486.
\end{align*}
Variant V.1 relies on the explicit computation of the matrix $D_k = I - \tfrac{1}{2}H_k + \tfrac{1}{10}H_k^2 - \tfrac{1}{120}H_k^3$, which has semi-bandwidth 3, and the solving of a linear system with $D_k$. 
For variant V.2 we can again make use of the fact that we have two conjugate complex vectors which contribute additively to $x_k$, so we need to solve just two tridiagonal systems of which the one using $\tau_1$  is complex and the other (belonging to $\tau_2$) is real. In contrast to V.1, the quality of $x_k$ is here influenced by the numerical accuracy of the underlying poles and coefficients.

\subsection{Arnoldi approximation of the exponential}

Since the analytical solution is given by $y(h) = \exp(hJQ)y_0$, one can also use the Arnoldi approximation directly for the matrix exponential and obtain the approximations
\[
x_k = V_k\exp(H_k)e_1 \|y_0\|_Q,
\]
with $H_k$ and $V_k$ from the $Q$-Arnoldi process (Algorithm~\ref{alg:arnoldi_P}) w.r.t.\ $\mathcal{K}_k(hJQ,y_0)$. Since $H_k$ is skew-symmetric, $\exp(H_k)$ is orthogonal, and thus, just as with the Pad\'e approximations, all the iterates $x_k$ preserve the $Q$-norm. There are two drawbacks which make us prefer the approach using the Pad\'e approximations: First, it is less clear when this iteration ought to be stopped and second, the approach requires to store all Arnoldi vectors $v_\ell$ (or to recurse to a two-pass strategy). This is in contrast to how we can implement the Arnoldi approximations in variant V.2, where the short-term recurrence in the $Q$-Arnoldi process can be used to construct a short-term recurrence update for the iterates $x_k$ in a similar way than is done in the SYMMLQ method for linear systems.

\section{Numerical results for a linear Poisson system}
\label{sec:num_lin}

We study the performance of the $Q$-Arnoldi iteration scheme for the Gauss integrators in comparison to the classical GMRES method. As benchmark example we consider a single mass-spring chain consisting of $N \in \mathbb{N}$ harmonic oscillators without input and dissipation elements as treated in \cite{GPBS2012}. This leads to a linear Poisson system $\dot{y} = JQy$ with
\begin{align*}
\renewcommand{\arraystretch}{0.5}
   J = \begin{bsmallmatrix}
         0  & 1 &\\
        -1  &  0 & \phantom{\tfrac{1}{m_1}} \\
             &  & 0 & 1 &  \\
             &  & -1& 0 & \phantom{\tfrac{1}{m_1}}\\
             &  &    &    & 0 &1\\
             &  &    &    & -1&0 & & \phantom{\tfrac{1}{m_1}}\\ 
             &  &    &    &     &  &  & \ddots \\
             &  &    &    &     &            &     \ddots    &  \phantom{\tfrac{1}{m_1}}\\
             &  &    &    &     &   &        &         & 0 & 1 \\
             & & & &  &  &  &  \phantom{\tfrac{1}{m_1}} &-1 & 0
    \end{bsmallmatrix},  \, \, \,
    Q = \begin{bsmallmatrix}
        k_1 &  & -k_1 & \\
         & \tfrac{1}{m_1} &  & \\
        -k_1 &  & k_1 + k_2 &  & -k_2 &  \\ 
                &  &                  &   \tfrac{1}{m_2} &  \\ 
                & &        -k_2           &  &  & \ddots \\
                & &                   &   & \ddots   & & - k_{N-1} &  \\
       & & & \ddots &   & \tfrac{1}{m_{N-1}}  &  &  & \\
       & & & & - k_{N-1}&  & k_{N-1} + k_{N} & \\
       & & & &  &  &  & \tfrac{1}{m_{N}}  
    \end{bsmallmatrix} \in \mathbb{R}^{2N \times 2N}.
\end{align*}
We particularly take equal masses $m_i=0.5$ and spring coefficients $k_i = 124$, $i = 1, \dots, N$, and set $y_0 = e_1$ as initial value. The system dimension scales with $N$, here we use $N=5000$.
  
As standard Krylov subspace method we apply the non-restarted GMRES method \cite{ZOU2023127869} to solve the linear system 
 \begin{equation} \label{eq:GMRES}
 \mathcal{D}_s(-hJQ)y_1 = \mathcal{D}_s(hJQ)y_0 
 \end{equation}
 for the Gauss collocation approximation $y_1$ of the exact solution $y(h)$, with the polynomial $\mathcal{D}_s$ from \eqref{eq:diagonalePadé}.
GMRES uses the Arnoldi process with the standard inner product ($Q = I$ in Algorithm \ref{alg:arnoldi_P})  to span the Krylov subspace $\mathcal{K}_k(\mathcal{D}_{s}(-hJQ),\mathcal{D}_{s}(hJQ)y_0)$. Note that we compute the matrix $\mathcal{D}_{s}(-hJQ)$ explicitly. The $k$-th iterate of GMRES is then obtained from minimizing the residual associated with \eqref{eq:GMRES} in the Euclidian norm over the $k$-th Krylov subspace. In GMRES the Krylov subspaces vary with the Gauss integrator selected, whereas in the $Q$-Arnoldi approximation (QAA) the Krylov subspace to be spanned is the same for all Gauss integrators, i.e.\ $\mathcal{K}_k(hJQ,y_0)$. Moreover, the upper Hessenberg matrix $H_k$ is in general fully occupied in GMRES, whereas it is tridiagonal, skew-symmetric in QAA. The implementation of the algorithms is realized in Python~3.12, using the packages \texttt{numpy} and \texttt{scipy}.

Our $Q$-Arnoldi process variants V.1 and V.2 might differ in computational effort depending on the chosen integrator. However, they provide comparable results in terms of convergence, energy preservation and long-time performance, which is why we only present the outcome of V.1 in the following.

\subsection{Convergence and energy preservation} 
\label{sec:conv}
\begin{figure}[t!]
\label{Fig:oneSteplinear}
  \centering
  \begin{subfigure}
      {\includegraphics[width=0.45\linewidth]{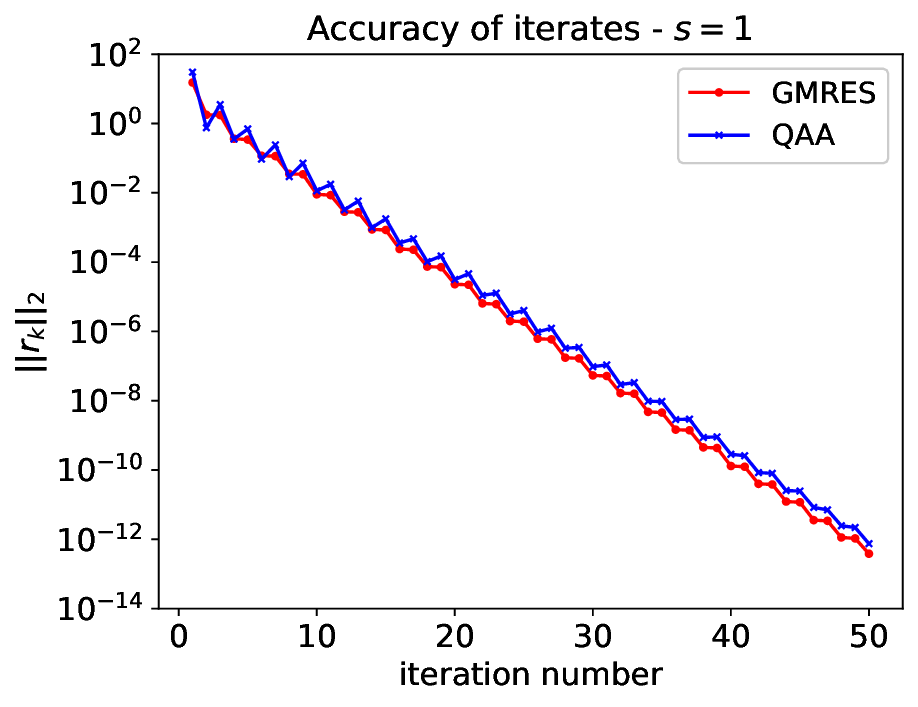}}%
  \end{subfigure}
  \hfill
 \begin{subfigure}
      {\includegraphics[width=0.45\linewidth]{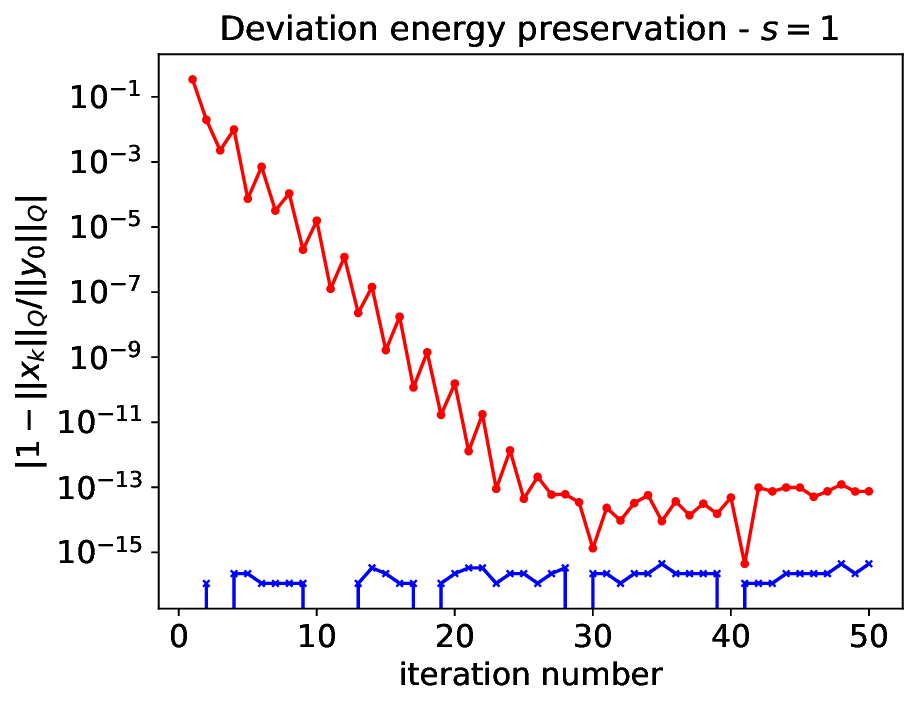}}%
  \end{subfigure}
  \begin{subfigure}
      {\includegraphics[width=0.45\linewidth]{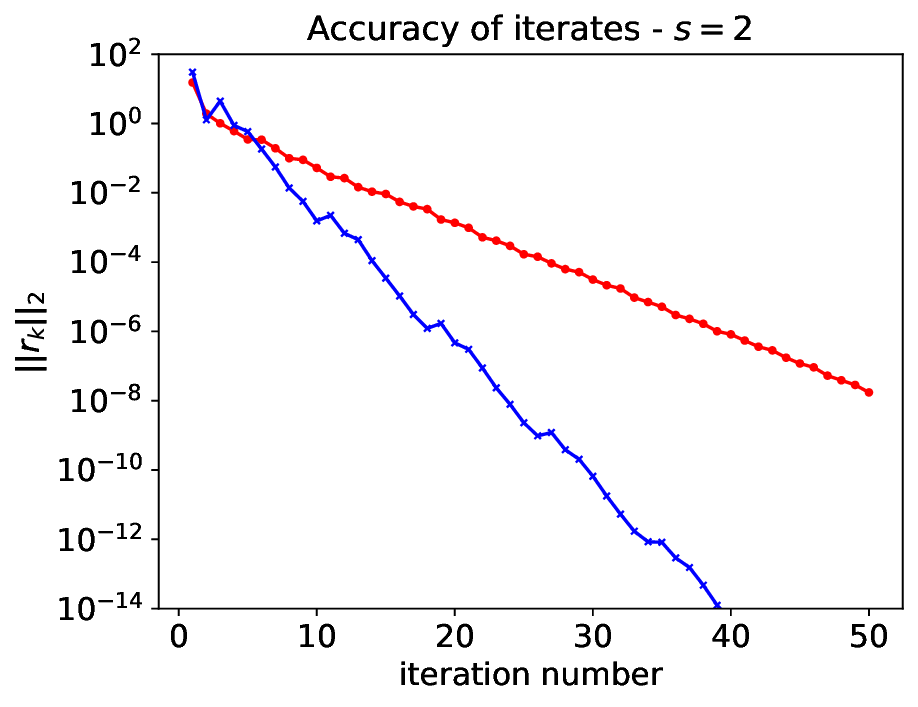}}%
  \end{subfigure}
  \hfill
 \begin{subfigure}
      {\includegraphics[width=0.45\linewidth]{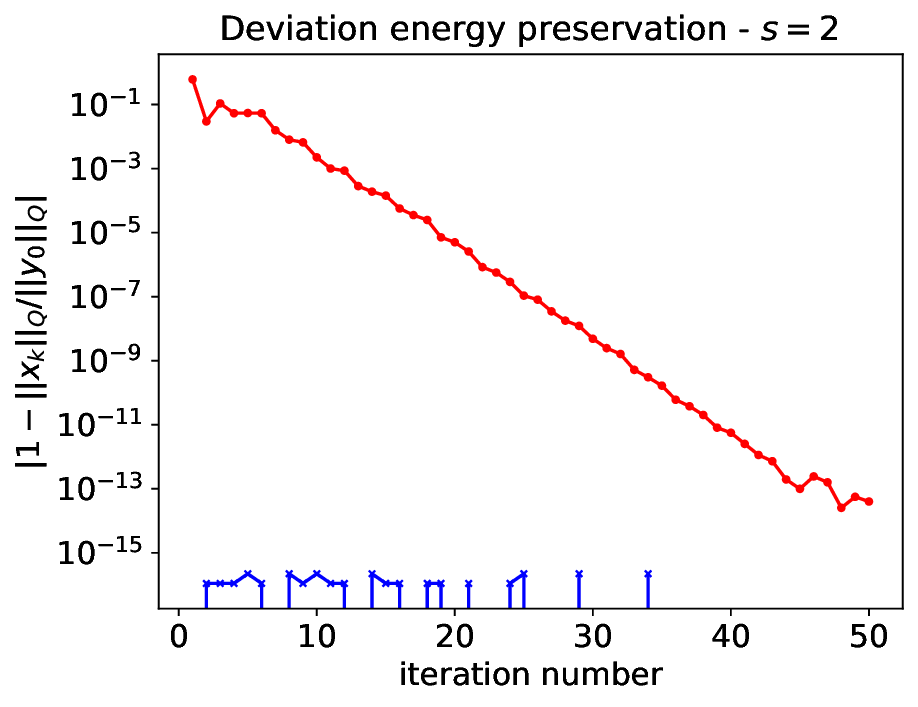}}%
  \end{subfigure}
  \begin{subfigure}
      {\includegraphics[width=0.45\linewidth]{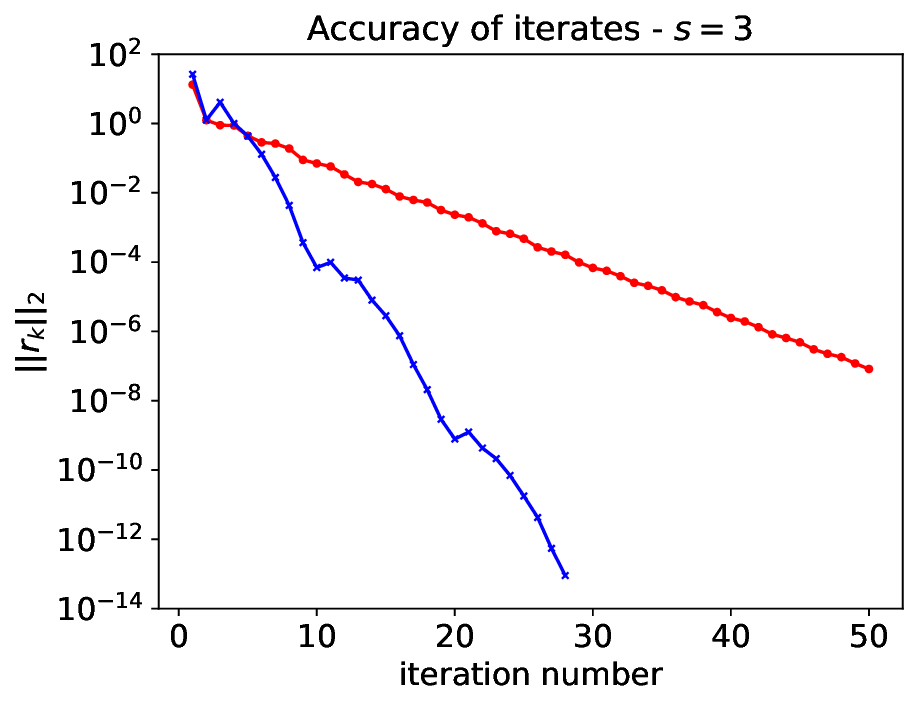}}%
  \end{subfigure}
  \hfill
 \begin{subfigure}
      {\includegraphics[width=0.45\linewidth]{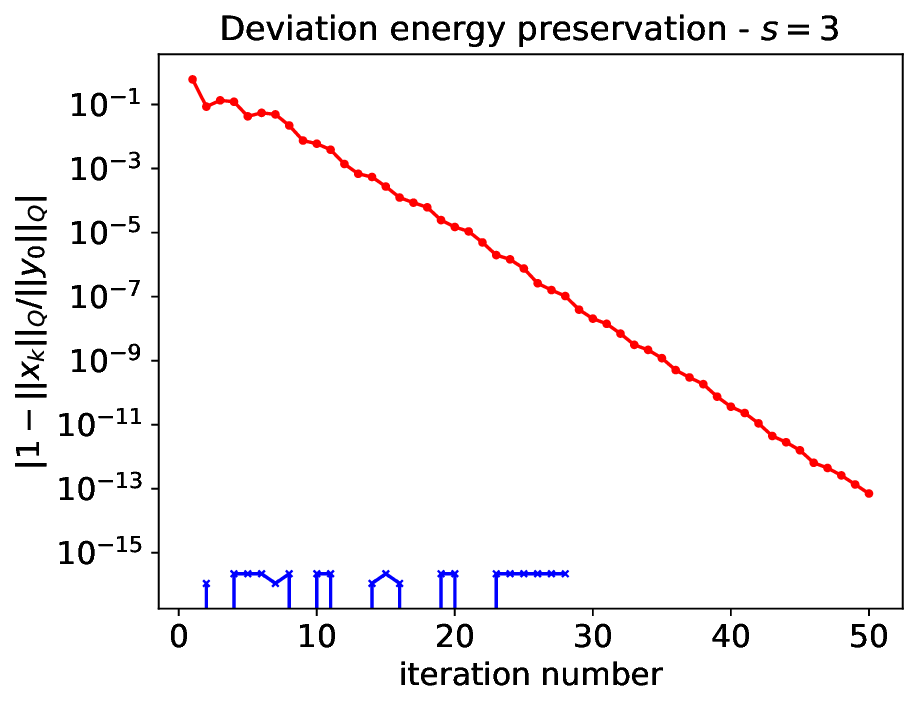}}%
  \end{subfigure}
  \caption{Performance of $Q$-Arnoldi approximation (QAA) and GMRES in the computation of $y_1$ for the mass-spring chain model. Left: Accuracy of iterates $x_k$ w.r.t.\ Euclidian norm of the residual $r_k = \mathcal{D}_{s}(-hJQ)x_k - \mathcal{D}_{s}(hJQ)y_0$. Right: Deviation from energy-conservation, i.e.\ $| 1 - \lVert x_k \rVert_Q/\lVert y_0 \rVert_Q|$. Top to bottom: Gauss integrators  with $s=1$, $s=2$ and $s=3$; $h = 0.1$.}
  \label{fig:midpointConvergence}
\end{figure}  

Focusing on the Gauss collocation approximation $y_1=\mathcal{R}_s(hJQ)y_0$ of the exact solution $y(h)$ for $s=1,2,3$, we compare the iterates $x_k$ of QAA and GMRES. Figure~\ref{Fig:oneSteplinear} illustrates the convergence behavior and the energy-preservation property of the iterates for $h=0.1$.
The accuracy of the iterates is visualized in terms of the Euclidian norm of the residual $r_k = \mathcal{D}_{s}(-hJQ)x_k - \mathcal{D}_{s}(hJQ)y_0$. For the energy preservation the relative error w.r.t.\ the initial value, i.e.\ $| 1 - \lVert x_k \rVert_Q/\lVert y_0 \rVert_Q|$, is shown. We observe a similar convergence behavior of QAA and GMRES for the midpoint rule ($s=1$), while for the other integrators ($s>1$) QAA converges much faster than GMRES. In addition, the higher $s$ the fewer iterations are required in QAA to achieve a certain accuracy. The effort of an iteration step is comparable in both iteration schemes. As designed, all Q-Arnoldi iterates conserve the energy for each $s$, the error is with $\mathcal{O}(10^{-15})$ of the order of the machine accuracy. GMRES, in contrast, is not tailored towards energy preservation. The first iterates violate the conservation of energy with a relative error of $\mathcal{O}(10^{-1})$, this error decreases as $k\rightarrow n$ ($n=2N$). Hence, conservation of energy requires the convergence of the scheme. This remains true even if one changes to a GMRES-variant using a non-standard inner product (see \cite{SS2024}, e.g.), where the $Q$-inner product would be the canonical candidate in our setting. Note that a larger system size $N$ or a larger step size $h$ usually leads to more iterations in both schemes (QAA and GMRES), but does not affect the qualitative behavior just discussed.

\subsection{Long-term performance}

 \begin{figure}[t!]
  \centering
  \begin{subfigure}
      {\includegraphics[width=0.33\linewidth]{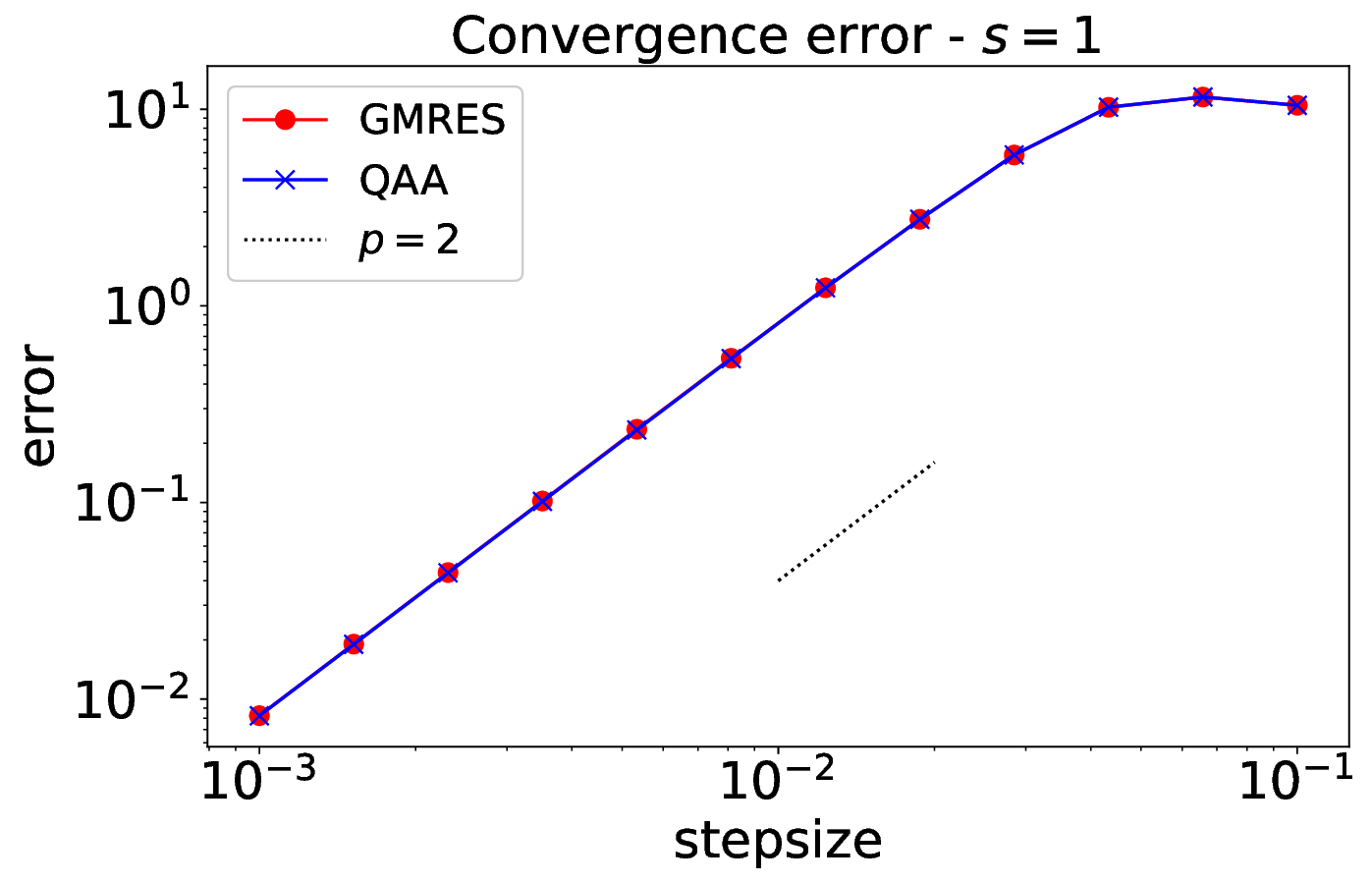}}%
  \end{subfigure}
 \begin{subfigure}
      {\includegraphics[width=0.335\linewidth]{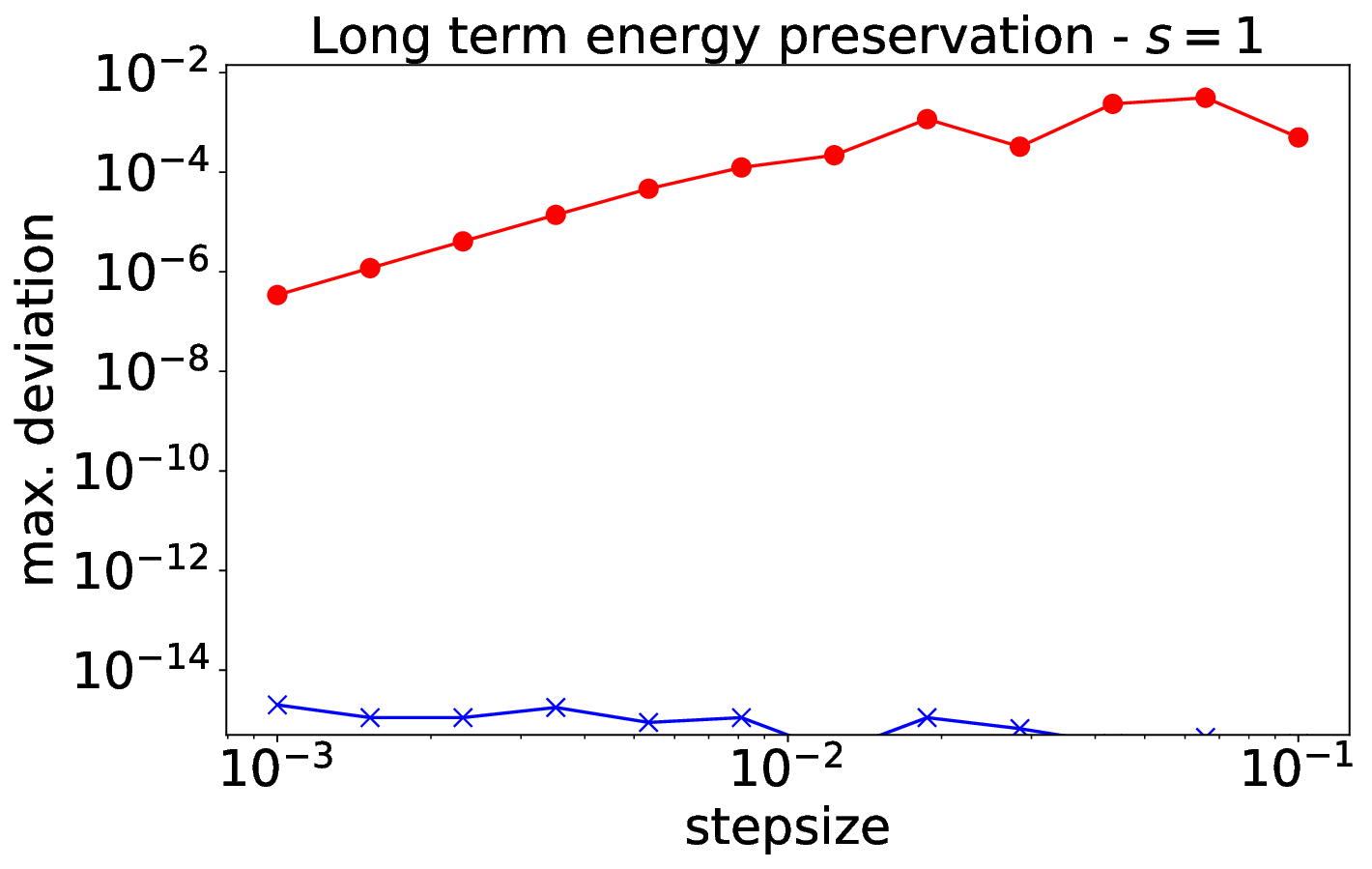}}%
  \end{subfigure}
  \begin{subfigure}
      {\includegraphics[width=0.32\linewidth]{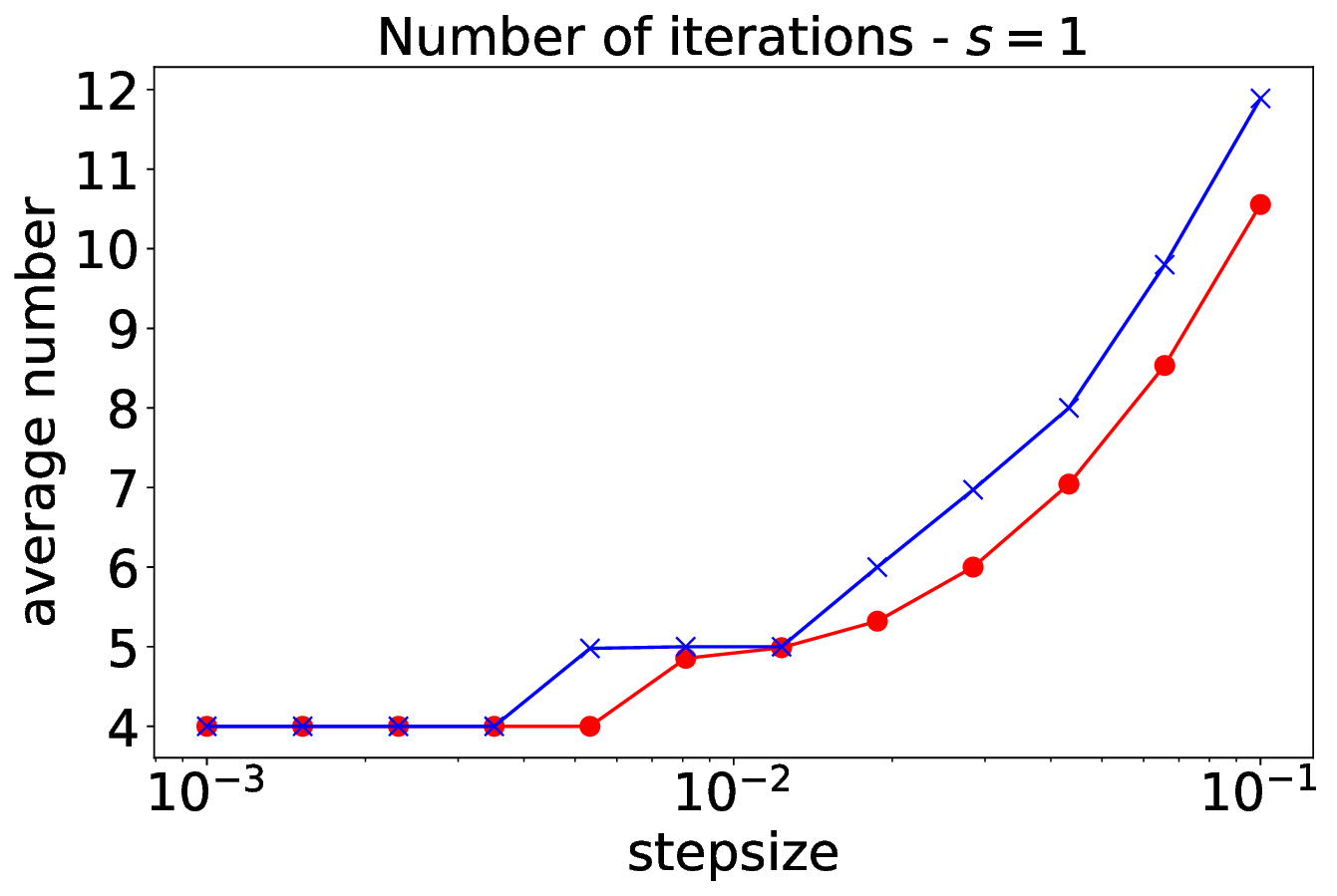}}%
  \end{subfigure}
\begin{subfigure}
      {\includegraphics[width=0.33\linewidth]{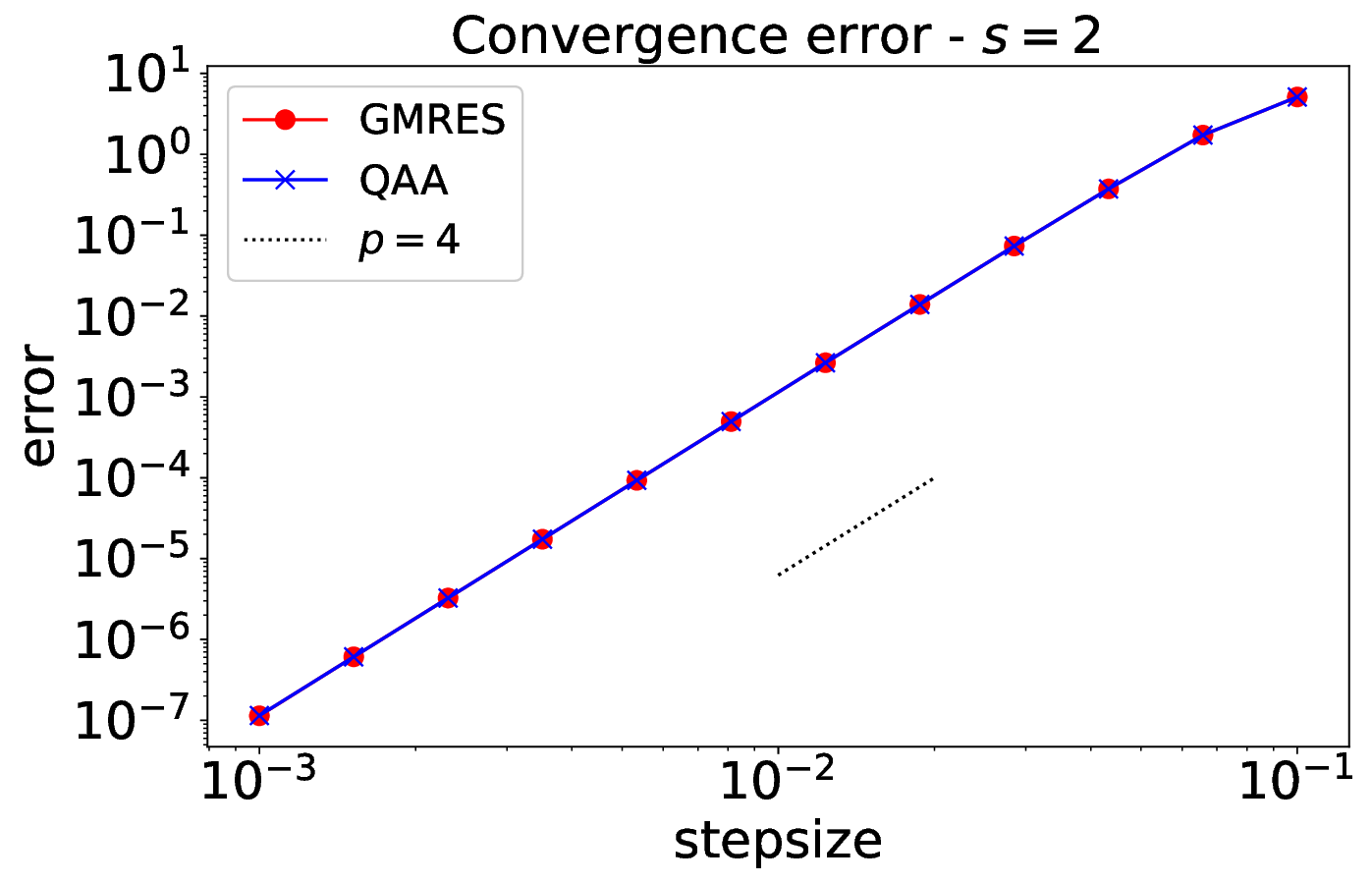}}%
  \end{subfigure}
 \begin{subfigure}
      {\includegraphics[width=0.335\linewidth]{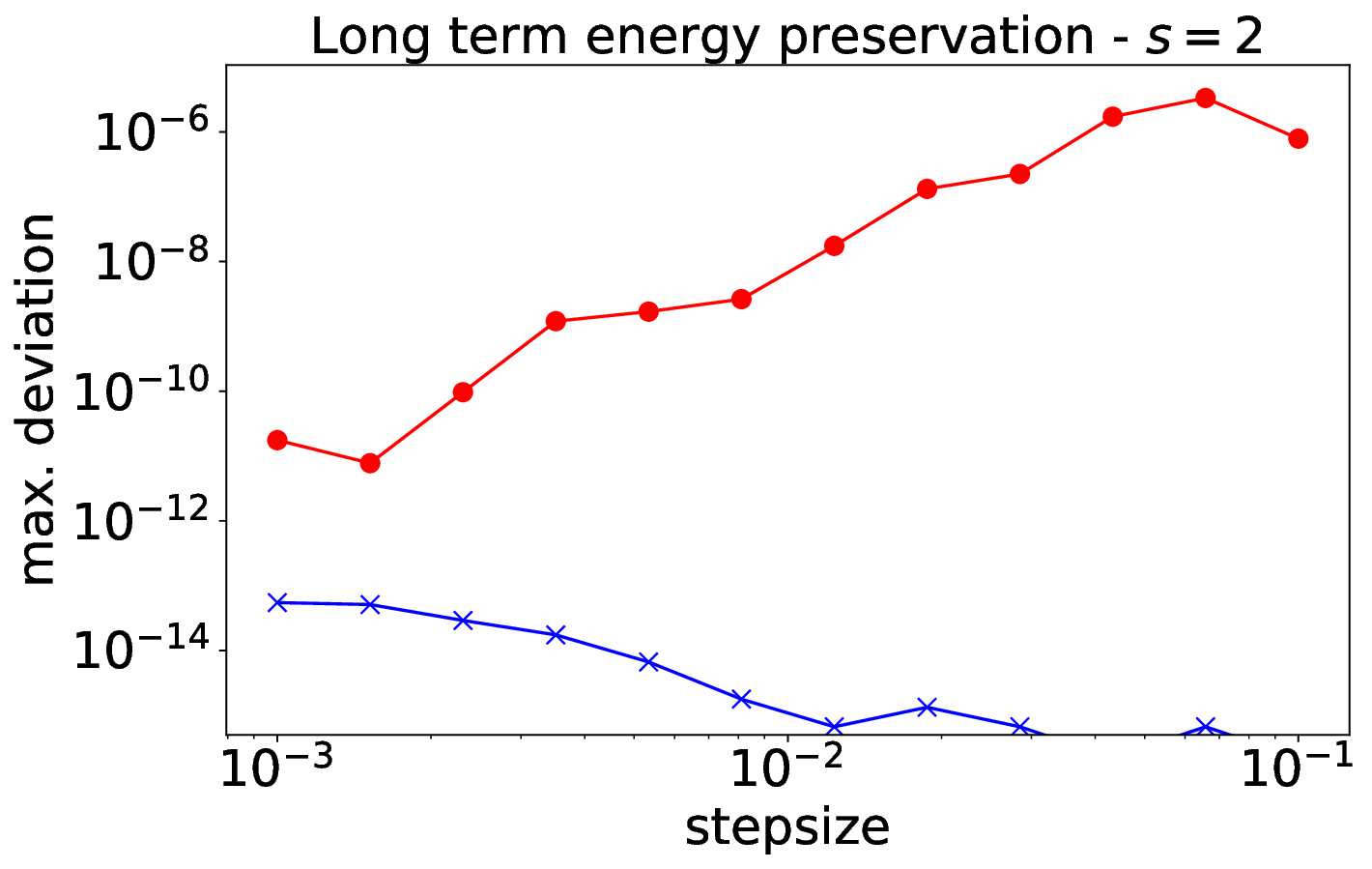}}%
  \end{subfigure}
  \begin{subfigure}
      {\includegraphics[width=0.32\linewidth]{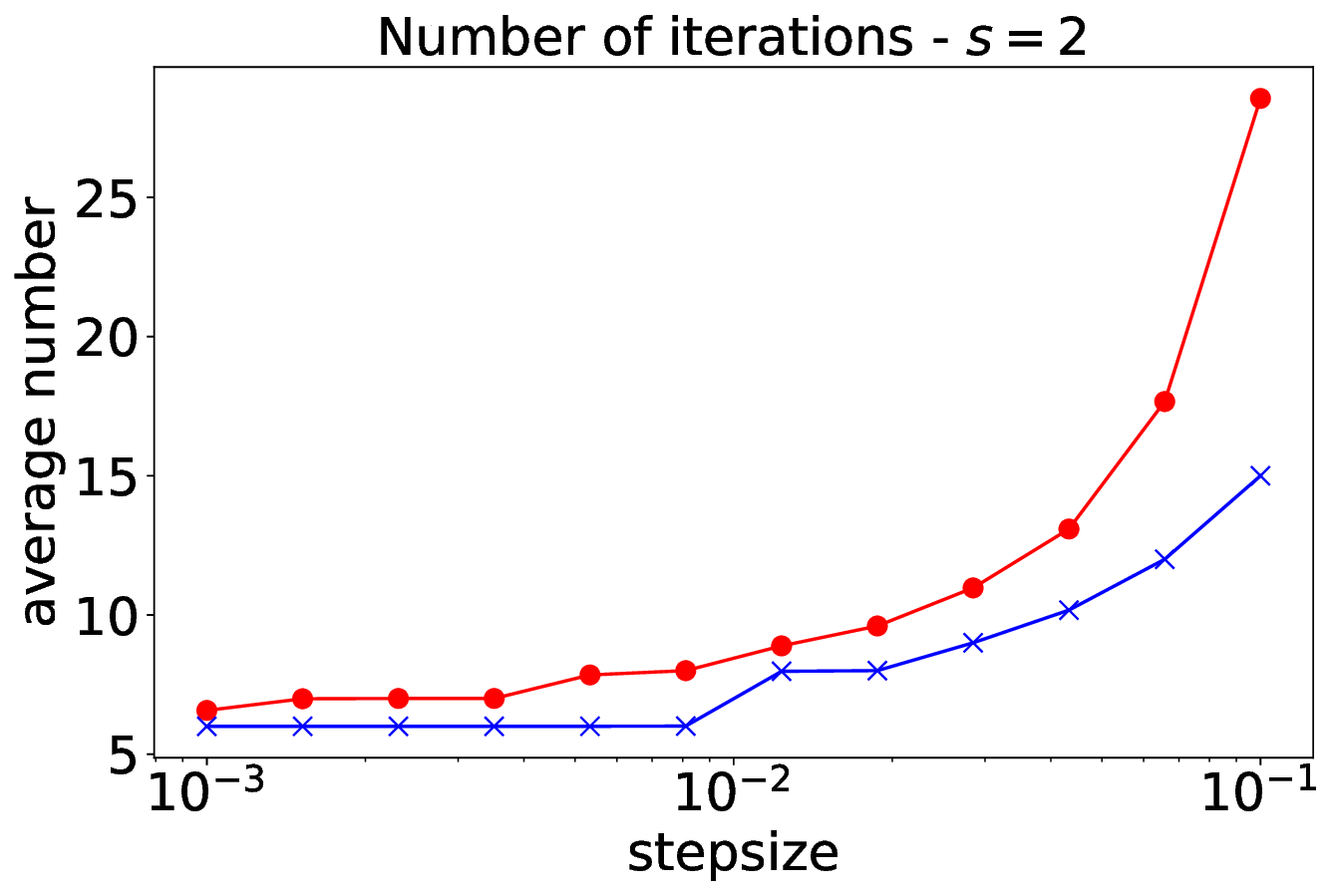}}%
  \end{subfigure} 
\begin{subfigure}
      {\includegraphics[width=0.33\linewidth]{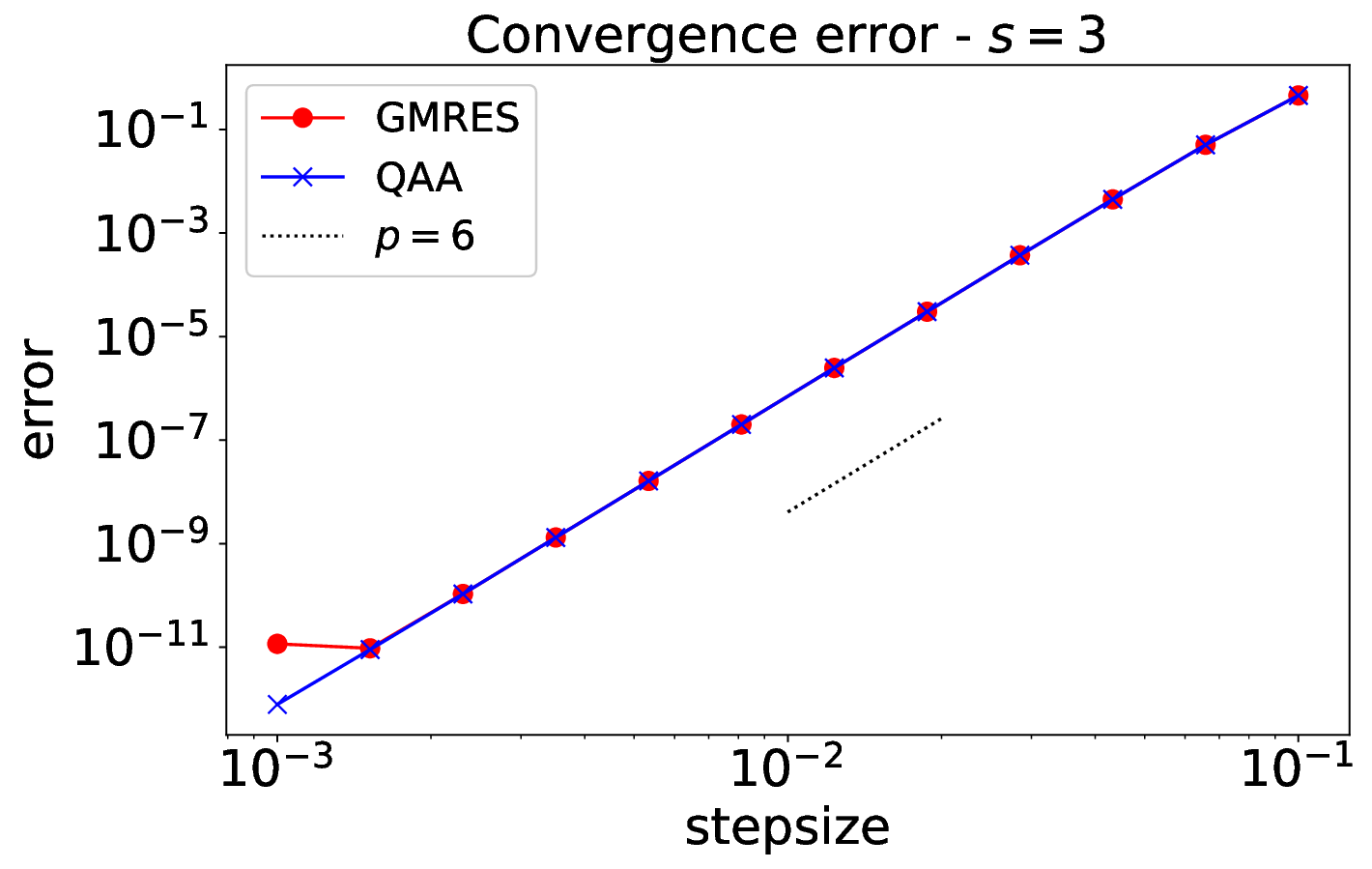}}%
  \end{subfigure}
 \begin{subfigure}
      {\includegraphics[width=0.335\linewidth]{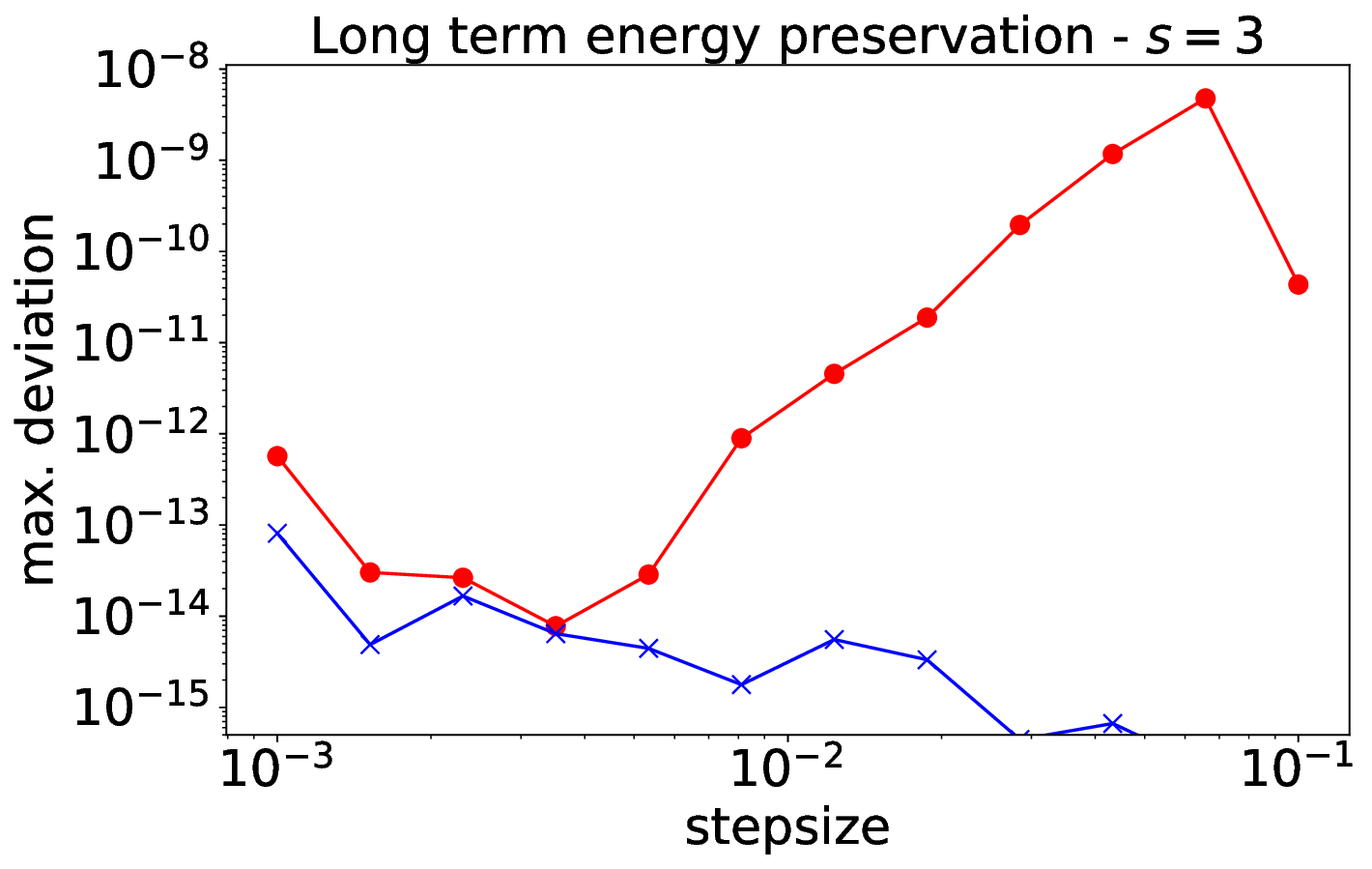}}%
  \end{subfigure}  
  \begin{subfigure}
   {\includegraphics[width=0.32\linewidth]{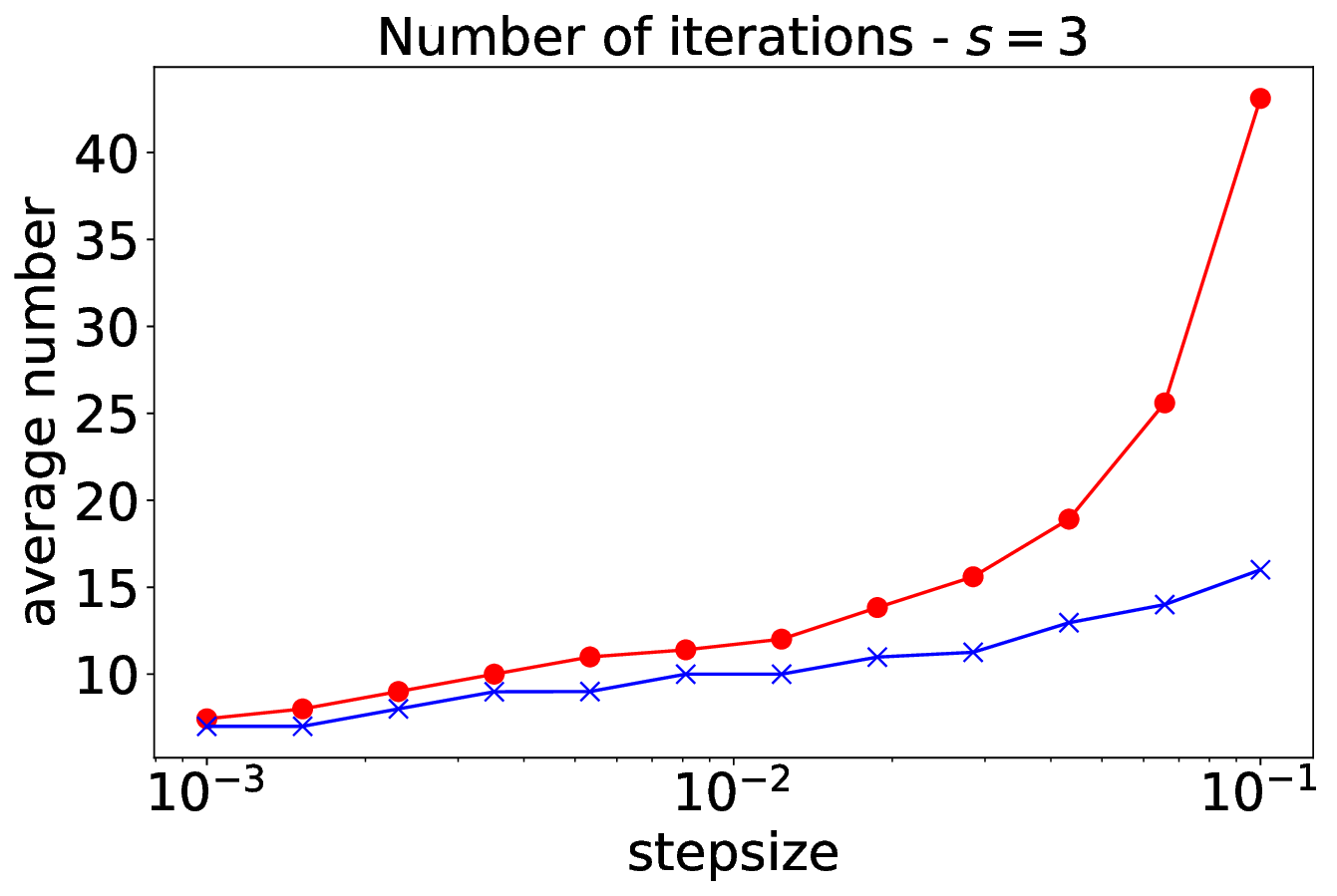}}%
  \end{subfigure}
  \caption{Performance of QAA and GMRES in Gauss integration of the mass-spring chain model over time interval $[0,1]$ for different step sizes $h$. Left: Convergence behavior of integrators w.r.t.\ $L^2([0,1])$-error. Middle: Absolute deviation from energy-conservation, i.e.\ $\max_i| 1 - \lVert y_i \rVert_Q/\lVert y_0 \rVert_Q|$. Right: Averaged number of Krylov subspace iterations per time step. Top to bottom: Gauss integrators with $s=1$, $s=2$ and $s=3$.}
\label{fig:orderMidpoint}
  \end{figure}

To explore the long-term performance, we solve the Poisson system on the time interval $[0,1]$ by help of the Gauss integrators using different step sizes $h$. In each time step $t_i=ih$, we iteratively compute the numerical solution $y_i\approx y(ih)$ with QAA and GMRES, respectively. We terminate the iteration once the Euclidian norm of the residual w.r.t.\ \eqref{eq:GMRES} is less than $h^{2s}$ as this is compatible with the order $p=2s$ of the respective integrator.\footnote{Due to the machine precision, the stopping criterion is set to $10^{-15}$ for $s=3$.} The reference solution is determined via the Python routine \texttt{scipy.sparse.linalg.expm\_multiply}, see \cite{numerikKonvergenzVgl} for details.

Figure~\ref{fig:orderMidpoint} illustrates the impact of the underlying iteration schemes on the numerical solution $y_i$ in terms of accuracy/convergence and energy preservation, and also highlights the respective costs. The overall convergence of the numerical solution to the reference as $h\rightarrow 0$ is achieved with both iteration schemes (QAA and GMRES); we numerically observe the expected convergence order $p=2s$ for all Gauss integrators (Fig.~\ref{fig:orderMidpoint}, left). With GMRES, energy preservation is not ensured in general. How well energy is preserved depends crucially on the step size $h$ and the order $p$ of the integrator: the smaller $h$ or the higher $p$, the smaller the deviation becomes (cf.\ Fig.~\ref{fig:orderMidpoint}, middle). This is to be expected, since the higher the accuracy demands, the closer the approximation $y_i$ gets to the reference $y(ih)$, which is energy-conserving. With QAA, in contrast, the numerical solution $y_i$ is energy-preserving over the entire time interval for all step sizes. Regardless of the Gauss integrator used we observe a maximum deviation from energy-conservation in the order of magnitude of the machine accuracy $\mathcal{O}(10^{-15})$ as in Section~\ref{sec:conv}. The marginal increase of the error to $\mathcal{O}(10^{-14})$ for smaller $h$ results from the accumulation of rounding errors which increase with increasing number of time steps. 

As for the cost of the computation, Fig.~\ref{fig:orderMidpoint} (right) shows the averaged number of iterations per time step for different step sizes. Both schemes (QAA and GMRES) show the tendency that the number of iterations that are required to meet the accuracy decreases as $h$ decreases. For $h=10^{-3}$, both schemes need 4-5 iterations independent of the Gauss integrator. For larger $h$ and $s>1$, QAA requires fewer iterations than GMRES. We note that the effort of an iteration step is comparable in both iteration schemes. Summing up, QAA is at least competitive with GMRES in terms of accuracy and cost, and superior in terms of energy conservation.

\section{Cayley Transformations in Non-linear Solvers} \label{sec:nonlinear}

If we have a state-dependent matrix $J=J(y)$, then the Poisson system is no longer linear. As a consequence, approaches for calculating a numerical solution that preserve the energy will naturally yield systems of non-linear equations. In this section, we discuss how we can construct non-linear iterative solvers for the systems arising from the midpoint rule such that they preserve energy in each iteration.

Applying the midpoint rule with step size $h$ to system \eqref{eq:energypresSys} leads to
\begin{align}
    y_1 &= y_0 + hJ(\tfrac{y_0 + y_1}{2})Q\tfrac{y_0 + y_1}{2} \nonumber \\
\Leftrightarrow    (I - \tfrac{h}{2}J(\tfrac{y_0 + y_1}{2})Q)y_1 &= (I + \tfrac{h}{2}J(\tfrac{y_0 + y_1}{2})Q)y_0 \nonumber \\
\Leftrightarrow   \hspace*{2.8cm}  y_1 &= (I - \tfrac{h}{2}J(\tfrac{y_0 + y_1}{2})Q)^{-1}(I + \tfrac{h}{2}J(\tfrac{y_0 + y_1}{2})Q)y_0.  \label{eq:IMP}
\end{align}
In the representation of the midpoint rule via the non-linear equation \eqref{eq:IMP} for $y_1$ we retrieve the first diagonal Padé approximation and the Cayley transform, since
\begin{align}
\label{eq:nonlinearCayley}
      y_1 = \mathcal{R}_1\big(hJ(\tfrac{y_0 + y_1}{2})Q\big)y_0 = \mathcal{C}\big(\tfrac{h}{2}J(\tfrac{y_0 + y_1}{2})Q\big)y_0=:\Phi_h(y_1).
\end{align}
Because of $\tfrac{h}{2}J(\tfrac{y_0 + y_1}{2})Q \in \mathfrak{g}_Q$, Theorem~\ref{th:LieGroupAlgebraPade} shows that the inverse of $I - \tfrac{h}{2}J(\tfrac{y_0 + y_1}{2})Q$ always exists and that, most importantly, energy is preserved just as it is along the solution $y(t)$ of \eqref{eq:energypresSys}, i.e.\ we have $\|y_1\|_Q = \|y_0\|_Q $. 
We now discuss two different approaches to iteratively solve the non-linear system \eqref{eq:IMP} for $y_1$.
 As in the linear case we are interested in iterative methods in which the energy is not only conserved upon convergence, but also at the level of each individual iterate. 
 
\subsection{Fixed-point iteration}
Equation \eqref{eq:nonlinearCayley} defines $y_1$ in fixed-point form as
\begin{equation*}
y_1 = \Phi_h(y_1) \enspace \text{ where }  
      \Phi_h(x) = \mathcal{C}\big(\tfrac{h}{2}J(\tfrac{y_0 + x}{2})Q\big)y_0.
\end {equation*}
We now investigate the convergence of the corresponding fixed-point iteration 
\begin{equation} \label{eq:fixed_point_iteration}
x_{k+1} = \Phi_h(x_k), \; k=0,1,\ldots \text{ with starting vector } x_0.
\end{equation}

\begin{theorem} \label{the:contraction}
Let $\mathcal{B}(y; r) = \{ x \in \mathbb{R}^n: \|x-y\|_Q \leq r\}$ be the ball centered at $y$ with radius $r \geq 0$ in the $Q$-norm. Then
\begin{itemize}
    \item[(i)] For any $x \in \mathbb{R}^n$ and any $h > 0$ we have $\Phi_h(x) \in \mathcal{B}(y_0; r_0)$ with $r_0 = 2 \|y_0\|_Q$.
    \item[(ii)] For any $x \in \mathbb{R}^{n}$ and $h > 0$ we have
    \[
    \|\Phi_h(x)\|_Q = \|y_0\|_Q.
    \]
    \item[(iii)] If $J(\cdot)Q$ is Lipschitz-continuous in the ball $\mathcal{B}(y_0; r_0)$ with Lipschitz constant $L$ w.r.t.\ the $Q$-norm, then 
    $\Phi_h$ satisfies
    \[
    \| \Phi_h(x) - \Phi_h(x') \|_Q \leq h\tfrac{Lr_0}{4}\, \| x - x'\|_Q \enspace \text{ for all } x, x' \in \mathcal{B}(y_0; r_0).
    \]
    \item[(iv)] Under the assumptions of (iii), if $h < \frac{4}{Lr_0}$ then $\Phi_h$ is a contraction on  $\mathcal{B}(y_0; r_0)$, and the iterates $x_k$ of $\eqref{eq:fixed_point_iteration}$ converge to the fixed point $y_1$ of $\Phi_h(x)$ for any starting vector $x_0 \in \mathbb{R}^n$. 
\end{itemize}
\end{theorem}

\begin{proof}
By Theorem~\ref{th:LieGroupAlgebraPade} we have $\mathcal{C}\big(\tfrac{h}{2}J(z)Q\big) = \mathcal{R}_1(hJ(z)Q) \in \mathcal{G}_Q$ for all $z\in \mathbb{R}^n$ which, by Lemma~\ref{lem:lie_algebra_spec}, implies $\|\mathcal{C}\big(\tfrac{h}{2}J(z)Q\big)\|_Q = 1$. Thus, for all $x \in \mathbb{R}^n$ and all $h \geq 0$  
    \begin{align*}
    \lVert \Phi_h(x) - y_0 \rVert_Q &= \lVert (\mathcal{C}\big(\tfrac{h}{2}J(\tfrac{y_0+x}{2})Q\big) - I)\,y_0 \rVert_Q
    \leq 
    \lVert \mathcal{C}\big( \tfrac{h}{2}J(\tfrac{y_0+x}{2})Q\big) - I \rVert_Q \lVert y_0 \rVert_Q \\
    &\leq (\lVert \mathcal{C}\big( \tfrac{h}{2}J(\tfrac{y_0+x}{2})Q \big) \rVert _Q + \lVert I \rVert_Q) \lVert y_0 \rVert_Q = 2 \lVert y_0 \rVert_Q = r_0.
    \end{align*}
    This proves (i). 
    
    To obtain (ii), we first note that for any $z \in \mathbb{R}^n$ and any $B$ in the Lie group $\mathcal{G}_Q$ we have
    \[
    \|Bz\|^2_Q = (Bz)^TQ(Bz) = z^T(B^TQB)z = z^TQz = \|z\|^2_Q. 
    \]
    Since $\tfrac{h}{2}J(\tfrac{y_0+x}{2})Q \in \mathfrak{g}_Q$, and since the Cayley transform $\mathcal{C}\big(\tfrac{h}{2}J(\tfrac{y_0+x}{2})Q\big)$ maps $\mathfrak{g}_Q$ on $\mathcal{G}_Q$, this shows $\|\Phi_h(x)\|_Q = \|y_0\|_Q$ for any $x$ and $h$. 
    
    To see (iii), we recall that the Cayley transform is Lipschitz continuous with Lipschitz constant $2$ on $\mathfrak{g}_Q$; cf.\ Corollary~\ref{cor:Lipschitz}. Hence, the Lipschitz continuity of $\Phi_h$ on $\mathcal{B}(y_0;r_0)$, $r_0=2\|y_0\|_Q$, follows from   
   \begin{align*}
    \| \Phi_h(x)-\Phi_h(x')\|_Q &\leq \| \mathcal{C}\big(\tfrac{h}{2}J(\tfrac{y_0+x}{2})Q\big) - \mathcal{C}\big(\tfrac{h}{2}J(\tfrac{y_0+x'}{2})Q\big)\|_Q \,\|y_0\|_Q \\
    &\leq 2 \| \tfrac{h}{2}J(\tfrac{y_0+x}{2})Q - \tfrac{h}{2}J(\tfrac{y_0+x'}{2})Q \|_Q \, \|y_0\|_Q \, \leq \,  \tfrac{h}{2}L \|y_0\|_Q \, \|x-x'\|_Q.    
    \end{align*}
    
    As for (iv), if $h < \tfrac{4}{Lr_0}$ we have $h\tfrac{Lr_0}{4} <1$, so that by (iii) the function $\Phi_h$ is indeed a contraction on $\mathcal{B}(y_0;r_0)$. Also note that by (i) we have $x_1 \in \mathcal{B}(y_0; r_0) $. With respect to the sequence of iterates $x_1,x_2,\ldots$ (in which we eliminated $x_0$), Banach's fixed-point theorem (see \cite{banach}, e.g.) ensures the convergence of the $x_k$ to the unique fixed point in $\mathcal{B}(y_0,r_0)$.
\end{proof}

A remarkable point in Theorem~\ref{the:contraction} is that we obtain {\em global} convergence of the iteration once $h\tfrac{Lr_0}{4} < 1$: we obtain convergence for just any starting vector $x_0$, irrespectively of its distance to the fixed point, while the Lipschitz-continuity needs to hold only locally (on $\mathcal{B}(y_0; r_0)$).  

For the fixed-point iteration \eqref{eq:fixed_point_iteration}, we have to compute a Cayley transform in each iteration. Thus, a further interesting fact is that if we do this iteratively using the $Q$-Arnoldi approximation approach from Section~\ref{subsec:KrylovArnoldi}, all iterates of this ``inner'' iteration preserve the $Q$-norm, too. As a consequence, whatever stopping strategy we take for the fixed-point iteration and for the inner iteration, we will always obtain energy-conserving approximations to $y_1$.

A practically feasible way to measure the quality of the iterates $x_k$ is to use the residual quantity given by 
\begin{equation} \label{eq:nonlinear_residual}
\widetilde{r}(x) = (I - \tfrac{h}{2}J(\tfrac{y_0 + x}{2})Q)(x-\Phi_h(x)) = (I - \tfrac{h}{2}J(\tfrac{y_0 + x}{2})Q)x - (I + \tfrac{h}{2}J(\tfrac{y_0 + x}{2})Q)y_0 
\end{equation}
and to stop the iteration when  $\widetilde{r}(x_k)$ is sufficiently small. This avoids the computation of inverse matrices or the solving of linear systems in the evaluation of $\Phi_h(x)$.

\subsection{Newton-like methods}
The fixed-point iteration \eqref{eq:fixed_point_iteration} converges at a linear rate. A general framework for obtaining faster converging methods is given by the Newton-like iterations.  In our setting \eqref{eq:nonlinearCayley}, we can use a  Newton-like scheme
to compute $y_1$ as a zero of the function
\[
F_h: \mathbb{R}^n \to \mathbb{R}^n, \enspace   x \mapsto x - \Phi_h(x).
\]

In a Newton-like method to obtain a root of a generic function $F:\mathbb{R}^n\rightarrow \mathbb{R}^n$, starting from an initial value $x_0 \in \mathbb{R}^n$, iterations are calculated as 
\begin{align}
\label{alg:NewLike}
x_{k+1} = x_k - B_k^{-1}F(x_k),
\end{align}
where $B_k$ is an approximation of the Jacobian $DF(x)$ of $F$ in $x = x_k$. If we choose $B_k = DF(x_k)$, we get the classical Newton method which, under canonical assumptions, converges locally and quadratically. In our situation, however, linear systems with $DF(x_k)$ are somewhat cumbersome to deal with practically, since $DF(x_k)$ involves a third order tensor coming from the derivative of $J(y)$ and two matrix inverses, see e.g.\ \cite{Hairer2004GeometricNI}. This is why we turn to the important class of Newton-like methods given by the quasi-Newton methods. There, $B_k$ is determined in an updating procedure exploiting the secant condition $B_{k+1} (x_{k+1} - x_k) = F(x_{k+1}) - F(x_k)$. One of the most well-known quasi-Newton methods is the BFGS method  \cite{dennis1977quasi} which we now consider in more detail as a representative of many other quasi-Newton methods. BFGS uses the update formula 
\begin{align*}
    B_{k+1}^{-1} = \left(I - \tfrac{s_k z_k^T}{z_k^T s_k} \right) B_{k}^{-1} \left(I - \tfrac{z_k s_k^T}{z_k^T s_k} \right) + \tfrac{s_k s_k^T}{z_k^T s_k}, \qquad B_0^{-1}=I,
\end{align*} 
where $z_k = F(x_{k+1}) - F(x_k)$ and $s_k = x_{k+1} - x_k$. 

For $F_h(x) = x -\Phi_h(x)$, 
the BFGS scheme \eqref{alg:NewLike} results in iterations of the form
\begin{equation}
\label{alg:NewLikeSpec}
\left.
\begin {array}{rcl}
w_k &= & \mathcal{C}\big(\tfrac{h}{2}J(\tfrac{y_0 + x_{k} }{2})Q\big)y_0 \\
x_{k+1} &=& x_k -  B_k^{-1}(x_k - w_k)
\end{array}
\right\} \enspace k=0,1,\ldots .
\end{equation}
Under canonical assumptions \cite{dennis1977quasi}, the BFGS iterates $x_k$ converges locally Q-superlinearly to the root $y_1$ of $F_h$, i.e.\ $\lim_{k\rightarrow \infty} \|x_{k+1}-y_1\|/\|x_k-y_1\|=0$. 
This trivially implies the R-superlinear convergence to $y_1$, i.e.\ there exists an error bounding sequence $\varepsilon_k$ such that $ \|x_k-y_1\|\leq \varepsilon_k$ for all $k$ and $\varepsilon_k$ converge Q-superlinearly to $0$. We refer to \cite{nocedal} for further details on R- and Q-convergence.

We on purpose expose $w_k$ explicitly in \eqref{alg:NewLikeSpec}, because this intermediate quantity plays a crucial role for energy preservation.  We call  $w_k$ the \emph{Cayley iterates} from now on. 
With respect to our main concern ---energy conservation---, we can now formulate the following result. 

\begin{theorem} \label{thm:cayley_newton_method}
Assume that $J$ is continuous on $\mathbb{R}^n$ and that the BFGS iterates $x_k$ from \eqref{alg:NewLikeSpec} converge to $y_1$ with $F_h(y_1) = 0$. 
Then 
\begin{itemize}
    \item[(i)] The Cayley iterates $w_k$ converge to $y_1$, too.
    \item[(ii)] For all $k$ we have $\|w_k\|_Q = \|y_0\|_Q$. 
    \item[(iii)] If $J$ is Lipschitz-continuous, then there exists a constant $C\geq0$ such that
    \[
    \|w_k-y_1\|_Q \leq C \|x_k-y_1\|_Q \,\, \text{ for all } k.
    \]
    In particular, if the iterates $x_k$ converge Q-superlinearly, do the Cayley iterates $w_k$.
\end{itemize}
\end{theorem}
\begin{proof}
Since $J$ and $\mathcal{C}$ are continuous, we have
\begin{align*}
    \lim_{k \to \infty} w_k = \lim_{k \to \infty} \mathcal{C}(\tfrac{h}{2}J(\tfrac{y_0 + x_{k} }{2})Q)y_0 = \mathcal{C}(\tfrac{h}{2}J(\tfrac{y_0+y_1}{2})Q)y_0 = y_1.
\end{align*}
This gives (i). 
And since $w_k = \Phi_h(x_k)$, we have  
$\|w_k\|_Q =  \|y_0\|_Q $ by  Theorem~\ref{the:contraction}(ii) which gives (ii). For (iii), let $L$ be a Lipschitz constant for  $J(\cdot)Q$. Since $y_1 = \mathcal{C}\big(\frac{h}{2}J(\frac{y_0+y_1}{2})Q\big)y_0$, we obtain, in a similar manner as in the proof of Theorem~\ref{the:contraction},
\begin{align*}
    \| w_k- y_1\|_Q &= \| \mathcal{C}\big(\tfrac{h}{2}J(\tfrac{y_0+x_k}{2})Q\big)y_0 - \mathcal{C}\big(\tfrac{h}{2}J(\tfrac{y_0+y_1}{2})Q\big)y_0\|_Q \\
    &\leq 2 \| \tfrac{h}{2}J(\tfrac{y_0+x_k}{2})Q - \tfrac{h}{2}J(\tfrac{y_0+y_1}{2})Q \|_Q \|y_0\|_Q\; \leq  \tfrac{h}{2}L \|y_0\|_Q\,  \|x_k-y_1\|_Q.    
    \end{align*}
\end{proof} 

In summary, Theorem~\ref{thm:cayley_newton_method} shows us that if the BFGS method converges for
$F_h$, then its Cayley iterates converge to the same solution and preserve the energy of the underlying Poisson system \eqref{eq:energypresSys}. And, as with the fixed-point iteration, if we use the $Q$-Arnoldi method as an inner iteration for approximating the Cayley transform, all these inner iterates also conserve energy. 
Thus, irrespectively of when we stop the inner or the BFGS iteration, our iterates $w_k$ preserve energy. As a stopping criterion for the BFGS iteration we can now use the residual quantity $\tilde{r}(w_k)$ with $\widetilde{r}$ from \eqref{eq:nonlinear_residual}. We refer to this approach as  the \emph{Cayley-BFGS method}.

\section{Numerical results for a non-linear system}
\label{sec:num_nl}

As a benchmark example for the non-linear case we consider the motion of a free rigid body (taken from \cite{gedrotierenderK}) which implies the non-linear Poisson system $\dot y=J(y)Qy$ with 
\begin{align*}
    J(y) = \left[\begin{array}{ ccc }
0 & -y_3 & y_2 \\
y_3 & 0 & -y_1 \\
-y_2 & y_1 & 0 \\
\end{array}\right], \qquad Q = \left[\begin{array}{ ccc }
\frac{1}{I_1} & 0 & 0 \\
0 & \frac{1}{I_2}  & 0 \\
0 & 0 & \frac{1}{I_3} 
\end{array}\right].
\end{align*}
The principal moments of inertia are set here to $I_1 = 2$, $I_2 = 1$ and $I_3 = \tfrac{2}{3}$, moreover we take $y_0 = [3,3,2]^T$ as initial value and $t\in [0,1]$. The discretization via the midpoint rule with step size $h$ leads to a non-linear system of equations (of the form \eqref{eq:nonlinearCayley}) in each time step.  In the following we numerically explore the performance of the fixed-point iteration (FP) and the Cayley-BFGS method in terms of accuracy, energy conservation and cost. As starting vector $x_0$ for the iterations at time point $t_{i+1}=(i+1)h$, we use the numerical solution $y_i$ of the former time step, i.e.\ $y_i\approx y(ih)$. The underlying computation of the Caley transform is realized with QAA variant~V.1 (as in Section~\ref{sec:num_lin}). Note, however, that for $n=3$ the use of a Krylov-subspace scheme is actually not necessary.

\subsection{Convergence and energy preservation} 
\label{sec:numNL_conv}
Considering the midpoint approximation $y_1=\mathcal{R}_1(hJ(\tfrac{y_0+y_1}{2})Q)y_0=\mathcal{C}(\tfrac{h}{2}J(\tfrac{y_0+y_1}{2})Q)y_0$ of the exact solution $y(h)$, we compare the fixed-point iterates and the Cayley iterates. Figure~\ref{fig:NewtonConvergence} illustrates the convergence behavior and the energy-preservation property for $h=0.1$. The accuracy of the iterates is visualized in terms of the Euclidian norm of the residual $\|\widetilde{r}(z_k)\|_2$ with $z_k=x_k$ in FP and $z_k=w_k$ in Cayley-BFGS. For the energy preservation the relative deviation w.r.t.\ the initial value, i.e.\ $| 1 - \lVert z_k \rVert_Q/\lVert y_0 \rVert_Q|$, is shown. In accordance with theory we numerically observe linear convergence of the fixed-point iteration. The Cayley iterates show a similar behavior in the first iterates, but then turn to superlinear convergence. While the quality of $x_8$ and $w_8$ is still comparable with a magnitude of $\mathcal{O}(10^{-5})$, the Cayley iterates achieve an accuracy of $\mathcal{O}(10^{-13})$ within the next 3 iterations, whereas the fixed-point iteration requires a total of 20 iterations. The significant acceleration is due to the fact that Cayley-BFGS uses Jacobian information, which is computed comparatively cheaply via the BFGS-update formula. As designed, both iteration schemes based on QAA are energy-conserving to machine precision, note that the logarithmic scaling in Fig.~\ref{fig:NewtonConvergence} (right) prevents exact zeros from being displayed. 

\begin{figure}[t!]
\centering
  \begin{subfigure}
      {\includegraphics[width=0.45\linewidth]{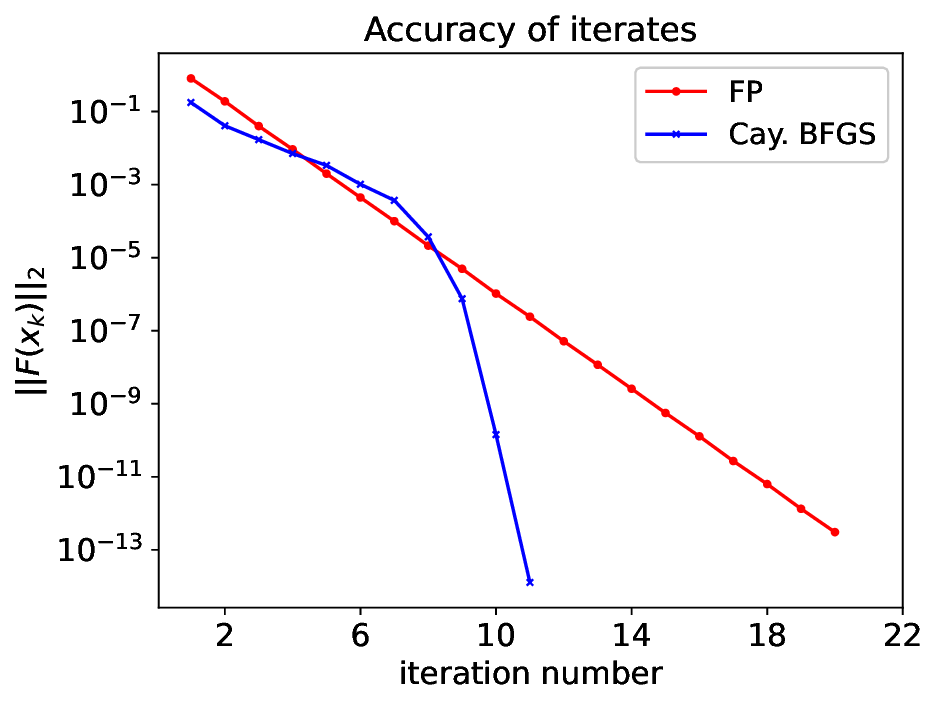}}%
  \end{subfigure}
  \hfill
 \begin{subfigure}
      {\includegraphics[width=0.45\linewidth]{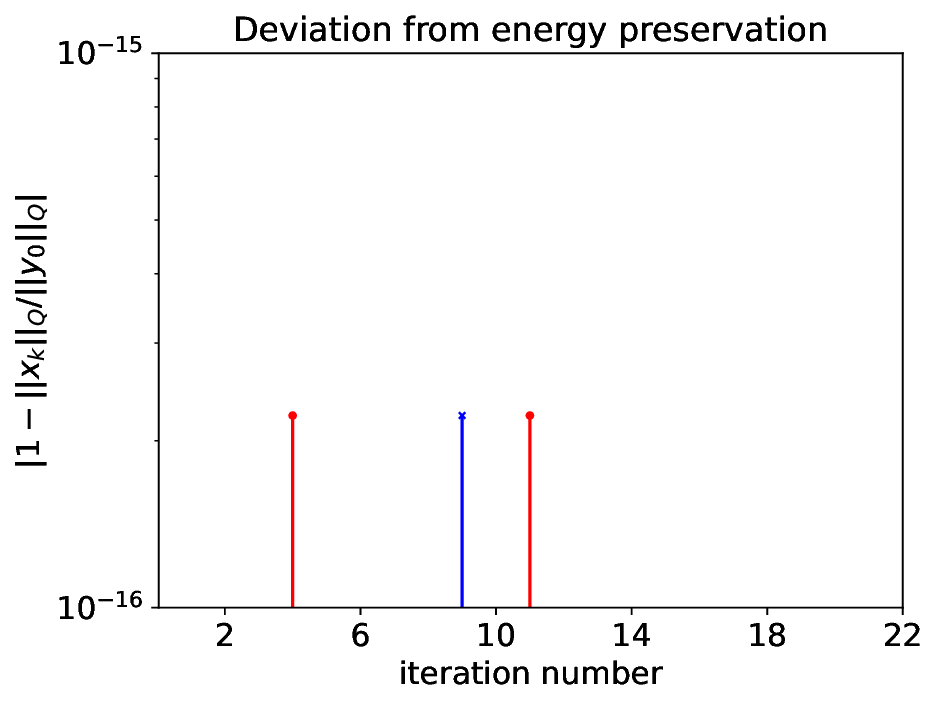}}%
  \end{subfigure}
\caption{Performance of  fixed point iteration (FP) and Cayley-BFGS method in the computation of $y_1$ for the rigid body model; $h=0.1$. Left: Accuracy of iterates w.r.t.\ Euclidean norm of the residual $\widetilde{r}$ from \eqref{eq:nonlinear_residual}.
Right: Deviation from energy-conservation.}
\label{fig:NewtonConvergence}
\end{figure}

\subsection{Long-term behavior}
To explore the long-term performance, we solve the Poisson system on the time interval $[0,1]$ using different step sizes $h$. In each time step $t_i=ih$, we iteratively compute the numerical solution $y_i\approx y(ih)$ with FP and Cayley-BFGS. We terminate the iteration once the Euclidian norm of the residual $\widetilde{r}$  is less than $h^2$ ---in consistency with the order of the midpoint rule. The reference solution is determined via the Python routine 
\texttt{scipy.integrate.solve\_ivp} using LSODA (Adams/BDF method with automatic stiffness detection and switching) with an absolute tolerance of $2\cdot 10^{-14}$, cf.\ \cite{petzold}.

Figure~\ref{fig:NewtonnOrder} illustrates the impact of the underlying iteration schemes and the stopping criterion on the numerical solution $y_i$ in terms of accuracy/convergence and energy preservation, and also highlights the respective costs. The convergence of the numerical solution to the reference is achieved with the used stopping criterion in both iteration schemes; we numerically observe the expected convergence order $p=2$ of the midpoint rule in the $L^2([0,1])$-norm as $h\rightarrow 0$ (Fig.~\ref{fig:NewtonnOrder}, left). Moreover, both schemes precisely preserve energy in the magnitude of machine precision. The increase of the deviation to $\mathcal{O}(10^{-12})$ for smaller $h$ results from the accumulation of rounding errors which increase with increasing number of time steps (Fig.~\ref{fig:NewtonnOrder}, middle). 
Concerning the computational cost, Fig.~\ref{fig:NewtonnOrder} (right) shows the averaged number of outer iterations per time step for different step sizes. The cost for an inner iteration with QAA is the same in both schemes. Both schemes have the tendency that the number of outer iterations that are required to meet the accuracy decreases with decreasing $h$ due to the improved starting vector $x_0$. On average, Cayley-BFGS needs less iterations than FP. However, because of the incorporated stopping criterion, the savings in iterations are significantly lower than expected from the results in Section~\ref{sec:numNL_conv}, where the schemes run until convergence.

  \begin{figure}[t!]
  \centering
  \begin{subfigure}
      {\includegraphics[width=0.33\linewidth]{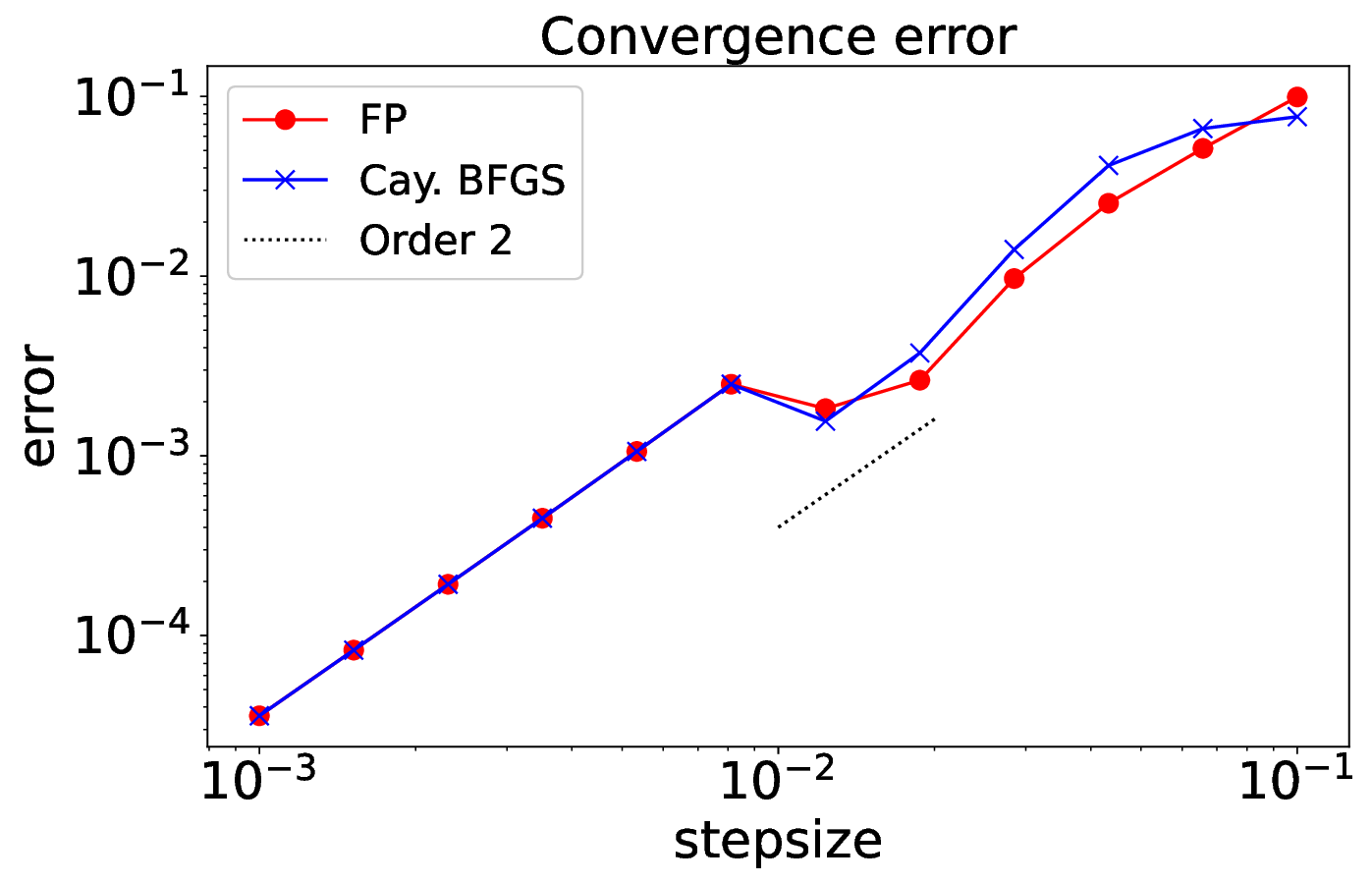}}%
  \end{subfigure}
 \begin{subfigure}
      {\includegraphics[width=0.335\linewidth]{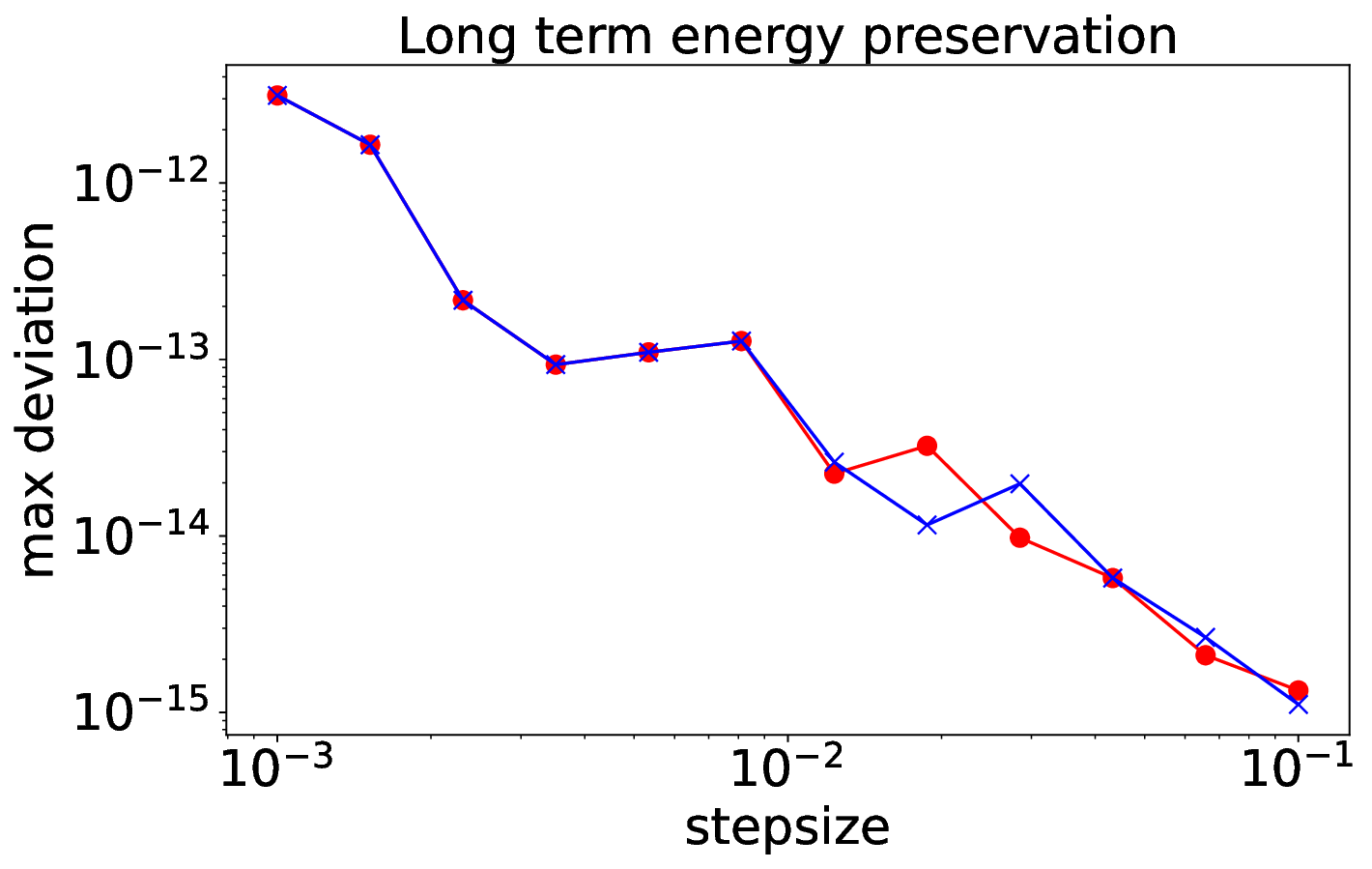}}%
  \end{subfigure}
\begin{subfigure}
      {\includegraphics[width=0.32\linewidth]{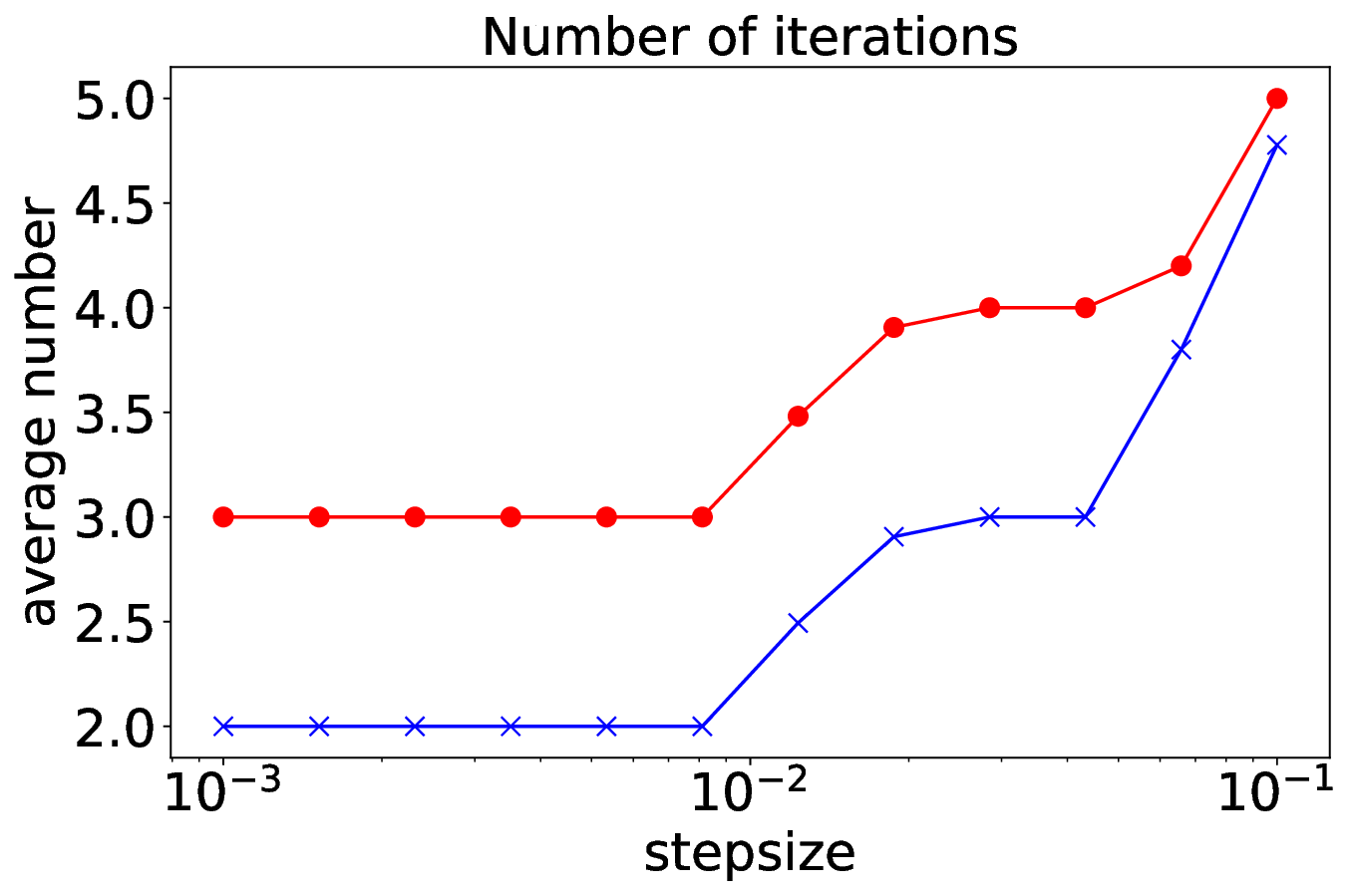}}%
  \end{subfigure}
\caption{Performance of FP and Cayley-BFGS in Gauss integration (midpoint rule) of the rigid body model over time interval $[0,1]$ for different step sizes $h$. Left: Convergence behavior of midpoint rule w.r.t.\ $L^2([0,1])$-error. Middle: Absolute deviation from energy-conservation, i.e.\ $\max_i| 1 - \lVert y_i \rVert_Q/\lVert y_0 \rVert_Q|$. Right: Averaged number of outer iterations per time step.}
\label{fig:NewtonnOrder}
\end{figure}

\section{Conclusion}

In this work, we have developed energy-preserving iteration schemes for the Gauss integration of Poisson systems. The proposed Krylov-subspace-based $Q$-Arnoldi approximation (QAA) for linear systems provides energy conservation not only at convergence, but also at the level of the individual iterates. Hence, the iteration can be terminated as soon as a desired accuracy is reached. QAA offers significant efficiency gains over standard iterative solvers used in implicit time integration and enables the efficient computation of high-dimensional problems and high-order approximations. In case of non-linear Poisson systems and a midpoint discretization, QAA can be cleverly embedded in non-linear solvers of fixed-point type as well as of Newton type. The resulting iteration schemes for the non-linear systems arising at each time step are of linear, superlinear and quadratic convergence order with comparable costs to the classical versions, but have energy-preserving iterates.


\begin{thebibliography}{10}

\bibitem{numerikKonvergenzVgl}
{\sc A.~H. Al-Mohyand and N.~J. Higham}, {\em Computing the action of the
  matrix exponential, with an application to exponential integrators}, SIAM
  Journal on Scientific Computing, 33 (2011), pp.~488--511.

\bibitem{arnol2013mathematical}
{\sc V.~I. Arnold}, {\em Mathematical Methods of Classical Mechanics},
  Springer, 2013.

\bibitem{Baker_Graves-Morris_1996}
{\sc G.~A. Baker and P.~Graves-Morris}, {\em {P}ad{\'e} Approximants},
  Cambridge University Press, 1996.

\bibitem{phDAE}
{\sc A.~Bartel, M.~Diab, A.~Frommer, M.~G{\"u}nther, and N.~Marheineke}, {\em
  Splitting techniques for {DAE}s with port-{H}amiltonian applications},
  Applied Numerical Mathematics, 214 (2025), pp.~28--53.

\bibitem{banach}
{\sc K.~Ciesielski}, {\em On {S}tefan {B}anach and {S}ome of his {R}esults},
  Banach Journal of Mathematical Analysis, 1 (2007), pp.~1--10.

\bibitem{crouch1993numerical}
{\sc P.~E. Crouch and R.~Grossman}, {\em Numerical integration of ordinary
  differential equations on manifolds}, Journal of Nonlinear Science, 3 (1993),
  pp.~1--33.

\bibitem{dennis1977quasi}
{\sc J.~Dennis and J.~Mor{\'e}}, {\em Quasi-{N}ewton methods, motivation and
  theory}, SIAM Review, 19 (1977), pp.~46--89.

\bibitem{diab2022flexibleshortrecurrencekrylov}
{\sc M.~Diab, A.~Frommer, and K.~Kahl}, {\em A flexible short recurrence
  {K}rylov subspace method for matrices arising in the time integration of
  port-{H}amiltonian systems and {ODE}s/{DAE}s with a dissipative
  {H}amiltonian}, BIT Numerical Mathematics, 63 (2023), p.~57.

\bibitem{frommer2023operatorSplitting}
{\sc A.~Frommer, M.~G{\"u}nther, B.~Liljegren-Sailer, and N.~Marheineke}, {\em
  Operator splitting for port-{H}amiltonian systems}, in Progress in Industrial
  Mathematics at {ECMI} 2023, Springer, 2025.
\newblock preprint arXiv:2304.01766.

\bibitem{FrSi06}
{\sc A.~Frommer and V.~Simoncini}, {\em Matrix functions}, in Model Order
  Reduction: Theory, Research Aspects and Applications, W.~H.~A. Schilders,
  H.~A. Vorst, and J.~Rommes, eds., Springer, 2008, pp.~275--303.

\bibitem{numericalAnaIntro}
{\sc W.~Gautschi}, {\em Numerical Analysis: An Introduction}, Birkh{\"a}user,
  1997.

\bibitem{GLMS2022}
{\sc C.~G{\"u}d{\"u}c{\"u}, J.~Liesen, V.~Mehrmann, and D.~B. Szyld}, {\em On
  non-{H}ermitian positive (semi)definite linear algebraic systems arising from
  dissipative {H}amiltonian {DAE}s}, SIAM Journal on Scientific Computing, 44
  (2022), pp.~A2871--A2894.

\bibitem{GPBS2012}
{\sc S.~Gugercin, R.~V. Polyuga, C.~Beattie, and A.~van~der Schaft}, {\em
  Structure-preserving tangential interpolation for model reduction of
  port-{H}amiltonian systems}, Automatica, 48 (2012), pp.~1963--1974.

\bibitem{Hairer2004GeometricNI}
{\sc E.~Hairer, C.~Lubich, and G.~Wanner}, {\em Geometric Numerical
  Integration: Structure Preserving Algorithms for Ordinary Differential
  Equations}, Springer, 2006.

\bibitem{hairer_siffODEquation}
{\sc E.~Hairer and G.~Wanner}, {\em Solving Ordinary Differential Equations II.
  Stiff and Differential-Algebraic Problems}, Springer, 1996.

\bibitem{matrixFunctBook}
{\sc N.~J. Higham}, {\em Functions of Matrices: Theory and Computation}, SIAM,
  2008.

\bibitem{jackiewiczCrouchGrossmanMethods}
{\sc Z.~Jackiewicz, A.~Marthinsen, and B.~Owren}, {\em Construction of
  {R}unge-{K}utta methods of {C}rouch-{G}rossman type of high order}, Advances
  in Computational Mathematics, 13 (2000), pp.~405--415.

\bibitem{lancaster1985theory}
{\sc P.~Lancaster and M.~Tismenetsky}, {\em The Theory of Matrices: With
  Applications}, Academic Press, 1985.

\bibitem{book_poisson_structures}
{\sc C.~Laurent-Gengoux, A.~Pichereau, and P.~Vanhaecke}, {\em Poisson
  Structures}, Springer, 2013.

\bibitem{lewis1994conserving}
{\sc D.~Lewis and J.~C. Simo}, {\em Conserving algorithms for the dynamics of
  {H}amiltonian systems on {L}ie groups}, Journal of Nonlinear Science, 4
  (1994), pp.~253--299.

\bibitem{KrylovPrinciplesAnalysis}
{\sc J.~Liesen and Z.~Strakos}, {\em {K}rylov Subspace Methods: Principles and
  Analysis}, Oxford University Press, 2012.

\bibitem{LopezApplCayleyapproach}
{\sc L.~Lopez and T.~Politi}, {\em Applications of the {C}ayley approach in the
  numerical solution of matrix differential systems on quadratic groups},
  Applied Numerical Mathematics, 36 (2001), pp.~35--55.

\bibitem{analysisProjMethod}
{\sc L.~Lopez and V.~Simoncini}, {\em Analysis of projection methods for
  rational function approximation to the matrix exponential}, SIAM Journal on
  Numerical Analysis, 44 (2006), pp.~613--635.

\bibitem{McLachlan_1998}
{\sc R.~I. McLachlan, G.~R.~W. Quispel, and N.~Robidoux}, {\em Unified approach
  to {H}amiltonian systems, {P}oisson systems, gradient systems, and systems
  with {L}yapunov functions or first integrals}, Physical Review Letters, 81
  (1998), p.~2399.

\bibitem{mclachlan1999geometric}
{\sc R.~I. McLachlan, G.~R.~W. Quispel, and N.~Robidoux}, {\em Geometric
  integration using discrete gradients}, Philosophical Transactions of the
  Royal Society A: Mathematical, Physical and Engineering Sciences, 357 (1999),
  pp.~1021--1045.

\bibitem{gedrotierenderK}
{\sc S.~Monaco, D.~Normand-Cyrot, M.~Mattioni, and A.~Moreschini}, {\em
  Nonlinear {H}amiltonian systems under sampling}, IEEE Transactions on
  Automatic Control, 67 (2022), pp.~4598--4613.

\bibitem{monch2024commutatorbasedoperatorsplittinglinear}
{\sc M.~M{\"o}nch and N.~Marheineke}, {\em Commutator-based operator splitting
  for linear port-{H}amiltonian systems}, Applied Numerical Mathematics, 210
  (2025), pp.~25--38.

\bibitem{MUNTHEKAAS1999115}
{\sc H.~Munthe-Kaas}, {\em High order {R}unge-{K}utta methods on manifolds},
  Applied Numerical Mathematics, 29 (1999), pp.~115--127.

\bibitem{nocedal}
{\sc J.~Nocedal and S.~J. Wright}, {\em Numerical Optimization}, Springer,
  2006.

\bibitem{A_Norton_2015}
{\sc R.~A. Norton, D.~I. McLaren, G.~R.~W. Quispel, A.~Stern, and A.~Zanna},
  {\em Projection methods and discrete gradient methods for preserving first
  integrals of {ODE}s}, Discrete and Continuous Dynamical Systems, 35 (2015),
  pp.~2079--2098.

\bibitem{norton2013discretegradientmethodspreserving}
{\sc R.~A. Norton and G.~R.~W. Quispel}, {\em Discrete gradient methods for
  preserving a first integral of an ordinary differential equation}, Discrete
  and Continuous Dynamical Systems, 34 (2014), pp.~1147--1170.

\bibitem{olver1986introduction}
{\sc P.~J. Olver}, {\em Applications of {L}ie Groups to Differential
  Equations}, Springer, 1986.

\bibitem{petzold}
{\sc L.~Petzold}, {\em Automatic selection of methods for solving stiff and
  nonstiff systems of ordinary differential equations}, SIAM Journal on
  Scientific and Statistical Computing, 4 (1983), pp.~136--148.

\bibitem{SogabeKrylovSubspaceMethods}
{\sc T.~Sogabe}, {\em {K}rylov Subspace Methods for Linear Systems: Principles
  of Algorithms}, Springer, 2022.

\bibitem{SS2024}
{\sc N.~Spillane and D.~B. Szyld}, {\em New convergence analysis of {GMRES}
  with weighted norms, preconditioning, and deflation, leading to a new
  deflation space}, SIAM Journal on Matrix Analysis and Applications, 45
  (2024), pp.~1721--1745.

\bibitem{PortHamIntrodcOverview}
{\sc A.~van~der Schaft and D.~Jeltsema}, {\em Port-{H}amiltonian Systems
  Theory: An Introductory Overview}, Now Publishers Inc, 2014.

\bibitem{varadarajan2013lie}
{\sc V.~S. Varadarajan}, {\em {L}ie Groups, {L}ie Algebras, and their
  Representations}, Springer, 2013.

\bibitem{wandeltGeoIntLieGrCayleyTrans}
{\sc M.~Wandelt, M.~G{\"u}nther, and M.~Muniz}, {\em Geometric integration on
  {L}ie groups using the {C}ayley transform with focus on lattice {QCD}},
  Journal of Computational and Applied Mathematics, 287 (2021), p.~112495.

\bibitem{ZOU2023127869}
{\sc Q.~Zou}, {\em {GMRES} algorithms over 35 years}, Applied Mathematics and
  Computation, 445 (2023), p.~127869.

\end{thebibliography}

\end{document}